\documentclass[a4paper,10pt]{amsart}
\usepackage[utf8]{inputenc}
\usepackage{amsthm}
\usepackage{amsmath}
\usepackage{bm}
\usepackage{amsfonts}
\usepackage{amssymb}
\usepackage{mathtools}
\usepackage{appendix}
\usepackage{mathrsfs}
\usepackage{setspace}
\usepackage{textgreek}
\usepackage{todonotes}
\usepackage{thmtools} 

\usepackage[pdfdisplaydoctitle,colorlinks,urlcolor=blue,linkcolor=blue,citecolor=blue]{hyperref} 

\newcommand{\C}{\mathbb{C}}

\newcommand{\F}{\mathcal{F}}

\newcommand{\M}{\mathcal{M}}
\newcommand{\N}{\mathbb{N}}

\newcommand{\PP}{\mathbb{P}}

\newcommand{\R}{\mathbb{R}}

\newcommand{\T}{\mathbb{T}}
\newcommand{\U}{\mathcal{U}}

\newcommand{\XX}{\mathcal{X}}
\newcommand{\x}{\mathbf{x}}

\DeclareMathOperator{\re}{Re}
\DeclareMathOperator{\im}{Im}
\let\div\relax
\DeclareMathOperator{\div}{div}
\DeclareMathOperator{\lip}{Lip}

\renewcommand{\epsilon}{\varepsilon}
\renewcommand{\setminus}{\smallsetminus}

\DeclareMathOperator{\imm}{i}

\let\div\relax
\DeclareMathOperator{\div}{div}

\newcommand{\set}[1]{\left\{#1\right\}}
\newcommand{\pa}[1]{\left(#1\right)}
\newcommand{\bra}[1]{\left[#1\right]}
\newcommand{\abs}[1]{\left|#1\right|}
\newcommand{\norm}[1]{\left\|#1\right\|}
\newcommand{\brak}[1]{\left\langle#1\right\rangle}

\newtheorem{thm}{Theorem}[section]

\newtheorem{definition}[thm]{Definition}
\newtheorem{cor}[thm]{Corollary}

\newtheorem{lemma}[thm]{Lemma}

\newtheorem{prop}[thm]{Proposition}

\theoremstyle{remark}
\newtheorem{rmk}[thm]{Remark}

\numberwithin{equation}{section}

\title[Burst point vortices]{Burst of point vortices and \\non-uniqueness of 2D Euler equations}
\author[F. Grotto]{Francesco Grotto}
  \address{Universit\'e du Luxembourg, Maison du Nombre, 6 Avenue de la Fonte, 4364 Esch-sur-Alzette, Luxembourg}
  \email{\href{mailto:francesco.grotto@uni.lu}{francesco.grotto@uni.lu}}
\author[U. Pappalettera]{Umberto Pappalettera}
  \address{Scuola Normale Superiore, Piazza dei Cavalieri 7, 56126 Pisa, Italia}
  \email{\href{mailto:umberto.pappalettera@sns.it}{umberto.pappalettera@sns.it}}
\date\today

\begin{document}

\begin{abstract}
We give a rigorous construction of solutions to the Euler point vortices system
in which three vortices burst out of a single one in a configuration of many vortices,
or equivalently that there exist configurations of arbitrarily many vortices
in which three of them collapse in finite time.
As an intermediate step, we show that well-known self-similar bursts
and collapses of three isolated vortices in the plane persist under a sufficiently regular external perturbation.
We also discuss how our results produce examples of non-unique weak solutions to 2-dimensional Euler's equations 
--in the sense introduced by Schochet-- in which energy is dissipated.
\end{abstract}

\maketitle

\section{Introduction}\label{sec:introduction}

The point vortices system is a classical model dating back to works of
Helmholtz \cite{Helmholtz1858,ArMe07}, for which many interesting open questions persist,
and whose role in the study of 2-dimensional incompressible fluid dynamics
has been and still is a determinant one.

The basic model is a Hamiltonian system of singular ODEs
describing the motion of $N$ points $z_1,\dots,z_N\in \R^2$,
each one being associated with an intensity $\xi_1,\dots,\xi_N\in\R\setminus\set{0}$,
\begin{equation}\label{eq:freevortices}
\dot z_j(t)=\sum_{k\neq j} \xi_k K(z_j(t)-z_k(t)),
\end{equation}
where $K(x)=\frac{1}{2\pi}\cdot \frac{x^\perp}{|x|^2}=-\nabla^\perp (-\Delta)^{-1}(0)$
is the Biot-Savart kernel on the plane $\R^2$, and $(x_1,x_2)^\perp=(-x_2,x_1)$.
Regarding positions $z_j$ as complex variables, the system has the equivalent formulation
\begin{equation}\label{eq:freevorticesC}
	\dot{\bar z}_j=\frac{1}{2\pi i} \sum_{k\neq j} \frac{\xi_k}{z_j(t)-z_k(t)}.
\end{equation}

The system is defined in such a way that the empirical measure $\omega_t=\sum_{j = 1}^N\xi_j z_j$
is (in a sense to be discussed below) a weak solution to 2-dimensional Euler's equations in vorticity form,
\begin{equation}\label{euler}
\begin{cases}
\partial_t \omega + u\cdot \nabla \omega =0,\\
\nabla^\perp \cdot u=\omega,
\end{cases}
\end{equation}
where, by Biot-Savart law, the second equation is equivalent to $u=-\nabla^\perp (-\Delta)^{-1} \omega$.

We will refer to the above system as the \emph{free} vortices system,
as opposed to vortices subject to external vector fields or in different geometries,
among which bounded domains of $\R^2$ with smooth boundary are of particular
interest for fluid-dynamics applications.
Literature on point vortices systems is extensive, and hard to summarise:
we refer to the monography \cite{Newton01} for analytic aspects
and complete references to works on integrable and non-integrable behaviour 
of $N$-vortices systems.
As for the fluid-dynamics point of view, we mainly refer to monographs \cite{MaPu94,BeMa02}:
in that context point vortices are usually studied in scaling limits
typical of statistical mechanics, as in classical works on Mean Field limits
\cite{Caglioti1992,Caglioti1995} or more recent Central Limit Theorems \cite{Flandoli2018,Grotto2020a}.

A crucial feature of system \eqref{eq:freevortices} is that
the driving vector field diverges only when $z_j(t)=z_k(t)$ for some $j \neq k$ and finite $t$,
while it is smooth otherwise.
Thus, for any initial configuration of distinct vortices, the system of $N$-vortices is locally
well-posed (in time), and solutions can be continued to the largest open interval $(a,b)\subset \R$
such that
\begin{itemize}
	\item either $b=+\infty$ or $\lim_{t \uparrow b}z_j(t)=\lim_{t \uparrow b}z_k(t)$ for some $j \neq k$,
	the latter being called a \emph{collapse} of vortices, and
	\item either $a=-\infty$ or $\lim_{t \downarrow a}z_j(t)=\lim_{t \downarrow a}z_k(t)$ for some $j \neq k$,
	the latter being called a \emph{burst} of vortices.
\end{itemize}
Collapses and bursts are referred to as singular solutions of \eqref{eq:freevortices}.
Besides translation and rotation invariance,
point vortices enjoy the following symmetry: solutions of \eqref{eq:freevortices} remain so
if time is inverted $t\mapsto -t$ and intensity signs are changed, $\xi_j\mapsto -\xi_j$.
This transformations turns collapses into bursts and vice-versa.

It is a classical result \cite{DuPu82,MaPu94} that for almost every initial configuration of vortices
--with respect to product measure on phase space $\C^N$--
there exists a unique global-in-time solution of the vortices system.
In some sense, the main result of the present contribution goes in the opposite direction,
showing that any given configuration of vortices can be the limit of a singular solution.

We will establish in \autoref{thm:Nvortices} the following:
\emph{given any configuration of $N$ distinct vortices on the plane at time $t=0$,
there exist a solution of \eqref{eq:freevortices} in a small time interval $(0,T)$
with $N+2$ vortices, three of which burst out of a single one from the initial configuration,
their intensities summing up to the split vortex one}.
By time inversion, this will imply existence of arbitrarily large configurations
in which three vortices collide in finite time.

Systems of three vortices are integrable, and self-similar bursts --equivalently, collapses--
of three vortices have been explicitly known since the XIX century.
In order to prove our main result, we will first show existence of bursts
of three vortices under the influence of an external vector field,
by expressing the system in a suitable coordinate system
describing ``closeness'' to the self-similar free solution,
and reformulating the problem as a fixed point.
This is an interesting result \emph{per se}: for instance we will deduce
from it the existence of bursts of three vortices in bounded domain.
Existence of a burst of three vortices out of one in a larger system of $N$
vortices will then follow from this preliminary result,
dividing the system into three bursting vortices under the ``external'' influence
of the other ones, and the rest of the configuration, involving only vortices
that do not collapse or burst.

Besides their interest in the context of integrable and non-integrable
behaviour of Hamiltonian system \eqref{eq:freevortices},
singular solutions of point vortices dynamics are relevant
also because of their connection with Euler's equations and fluid dynamics.
Possible implications of our result are manifold,
and in the last part of the paper, \autoref{sec:weaksol}, we will discuss in particular
how point vortices systems including bursts and collapses can be
rigorously interpreted as weak solutions of Euler's equation,
and how our main result can be read as a strong non-uniqueness result for such solutions.

In the case of non-singular solutions, a series of works \cite{MaPu83,Marchioro88,Marchioro1998,BuMa18,IfMa18,LoWaZe19}
has substantiated the idea that point vortices are true (although irregular) solutions
of Euler's equation, by obtaining the former as limits of $L^\infty$ vorticity patches
solving Euler's equation\footnote{It is a classical result that 2D Euler's equations are
well posed for $L^\infty$ initial data, \cite{Judovic63}.}.
In fact, as first described in \cite{Schochet96},
a system of collapsing vortices satisfies the same weak formulation
of Euler's equations of a non-collapsing system, provided that we postulate that
collapsing vortices merge into a single one, whose intensity is the sum of merging ones.

Due to the possibility of causing a burst of vortices from an arbitrary configuration,
our notion of weak solution is necessarily non-unique,
uniqueness naturally being recovered if bursts and collapses are excluded \emph{a priori}.
The idea of a connection between singular vortices evolution
and non-uniqueness for fluid-dynamics equations dates back to \cite{Novikov1979},
and our result makes this insight rigorous.

Non-uniqueness can also be exploited to produce \emph{intrinsically stochastic}
weak solutions of Euler's equations, that is Markov processes
whose trajectories solve the (deterministic!) Euler's equations in weak sense.
To the best of our knowledge, this is a novelty in this context,
so we briefly discuss it below (\autoref{ssec:intrinsic}) and leave it for future developments.
The results of this paper might also contribute to answer to
the question of whether the Liouville operator
(the generator of Koopman's group of unitaries)
associated to the $N$ point vortices evolution
is essentially self-adjoint on certain classes of observables,
a problem left open in \cite{Albeverio2003,Grotto2020}.
We also mention \cite[Section III]{aref10}, in which existence of arbitrarily
large collapsing configurations of vortices was left open,
and where some interesting consequences are conjectured.

A role diametrically opposed to singular configurations
is played by relative equilibrium configurations,
in which point vortices evolve while their relative positions remain fixed.
This is in fact a very active area of research:
we limit ourselves to refer to some recent works \cite{Newton07,Newton14,GeOr19,Gebhard18,Qun19}
and references therein.
For the sake of completeness, let us also recall that configurations leading to
\emph{quasi-}periodic vortex dynamics have been proved to form a subset
of positive measure of phase space in \cite{Khanin81}, by means of KAM theory.

It is worth mentioning that absence of collisions for any initial configuration
can be obtained by subjecting vortices to an additional stochastic forcing,
as discussed already in \cite{DuPu82} and, in greater generality, \cite{Flandoli2011},
the latter with emphasis on the connection with stochastic 2-dimensional Euler equations.
Whether a zero-noise limit can identify the ``correct'' continuation
of vortices motion after collapse is a very interesting although completely open problem.
This question is in fact suggested by an important analogue of point vortices:
Vlasov-Poisson charges in dimension 1. Indeed, analogously to weak solutions discussed below in \autoref{sec:weaksol},
in that context measure-valued solutions are not unique, and point charges can transform
into electron sheets (\emph{cf.} \cite[Section 13]{BeMa02}),
but uniqueness is recovered when including a suitable stochastic forcing, \cite{Delarue2014}.

The article is organised as follows. In \autoref{sec:burstinexternal}
we will establish existence of bursts of three vortices in an external field and
in a bounded domain. In \autoref{sec:burstN} we will complete the proof of our main result,
\autoref{thm:Nvortices}. Finally, \autoref{sec:weaksol}
discusses point vortices as weak solutions of Euler's equation, non-uniqueness of 
the latter and intrinsic stochasticity, and is concluded by some remarks on
dissipation of energy in singular vortices dynamics.

\subsection{Notation}

We denote as usual by $C^k(\R^2,\R^2)$ the Banach space of $k$ times differentiable  functions $f:\R^2 \to \R^2$
with continuous $k$-th derivative, $k \in \N$, endowed with the norm
\begin{equation*}
	\norm{f}_{C^k}=\norm{f}_\infty+\norm{Df}_\infty+\cdots +\norm{D^k f}_\infty.
\end{equation*}
As we just did, we will use $\norm{\cdot}_\infty$ for the supremum norm in different function spaces:
the correct meaning will always be deducible from the context.
The same holds for functions having as domains open sets of $\R^2$.
In the following, we will often identify $\R^2=\C$: this is to make use of complex symbolism for polar coordinates,
and differentiation (in space variables) will always be intended as real.
The same identification will be used for $C^k(\R^2,\R^2) = C^k(\C,\C)$. 

For a given time interval $I \subset \R$, we denote by $\lip\pa{I,X}$ the space of Lipschitz functions on $I$
taking values in a Banach space $X$, and denote $\bra{f}_{\lip\pa{I,X}}$ the Lipschitz seminorm
\begin{equation*}
	\bra{f}_{\lip\pa{I,X}}=\sup_{\substack{s,t \in I \\ s \neq t}}\frac{\norm{f(t)-f(s)}_X}{|t-s|}.
\end{equation*}
Analogously, for $0<\alpha<1$ let 
\begin{equation*}
	\bra{f}_{C^\alpha\pa{I,X}}=\sup_{\substack{s,t \in I \\ s \neq t}}\frac{\norm{f(t)-f(s)}_X}{|t-s|^\alpha}
\end{equation*}
denote the H\"older seminorm on the space $C^\alpha\pa{I,X}$
of $X$-valued $\alpha$-H\"older functions.
When no ambiguity may occur, we simply denote
\begin{equation*}
	\bra{f}_{\lip}=\bra{f}_{\lip\pa{I,X}}, \quad \text{ and }\quad \bra{f}_{C^\alpha}=\bra{f}_{C^\alpha\pa{I,X}}.
\end{equation*}

For any $T>0$ let us also introduce the space
\begin{align*}
E_T = C([0,T), C^2(\C,\C)) \cap
\lip ([0,T), C(\C,\C))
\end{align*}
of functions $f$ of time and space, endowed with the norm
\begin{align*}
\norm{f}_{E_T} = \sup_{t \in [0,T)} \norm{f_t}_{C^2} + \bra{f}_{\lip}.
\end{align*}
In our main results concerning existence of bursts of vortices under the influence of external fields, we will always take fields belonging to the space $E_T$ just introduced.
This particular choice is due to the fact that in this way we are able to construct a \emph{continuous} map from $E_T$ to bursts of vortices, see \autoref{prop:dependenceonf_z} below.
This will be of fundamental importance in other parts of this paper, see for instance \autoref{prop:existenceboundeddomain} and \autoref{thm:Nvortices}. 

Finally, in a $n$-fold product $X^n$ of a space $X$, we denote by
\begin{equation*}
	\triangle^n=\set{(x_1,\dots,x_n)\in X^n:\, x_i=x_j \, \text{ for some } i\neq j}
\end{equation*}
the generalised diagonal. We will use the symbol for diagonals of different products:
its meaning will always be clear from the context, as in expressions like $X^n\setminus \triangle^n$.

\section{Three-vortices Bursts in an External Field}\label{sec:burstinexternal}

The system of three vortices on the whole $\C$ is integrable,
and it is well known that it 
admits singular solutions in which the vortices collapse in --or burst out of-- 
a single point, with self-similar trajectories.
In fact, singular solutions of a system of three vortices have been completely characterised:
we refer to \cite{AeRoTh92,aref10,KrSt18} for the following result.

\begin{prop}
	For $N=3$, the system \eqref{eq:freevorticesC} admits a solution in which three vortices
	with intensities $\xi_1,\xi_2,\xi_3\in\R \setminus \{0\}$ starting from distinct positions $z_1,z_2,z_3\in\C$
	collapse at finite time if and only if
	\begin{equation}\label{eq:condcollapse}
	\sum_{j \neq k}\xi_j\xi_k=0,\quad I=\sum_{j \neq k}\xi_j\xi_k |z_j-z_k|^2=0.
	\end{equation}
\end{prop}

We recall that $I(t)=I$ defined in the latter equation is a first integral of motion,
the \emph{moment of inertia}. Although the results only mentions collapses,
by changing signs to intensities $\xi_j$ and reversing time, a collapse of vortices
can be turned into a burst, and vice-versa. 
Let us also remark that the first condition in \eqref{eq:condcollapse}
excludes that collapsing vortices, and thus bursts, have intensities of the same sign:
in fact in that case the Hamiltonian 
\begin{equation*}
	H=-\frac{1}{2\pi} \sum_{j\neq k} \xi_j \xi_k \log |z_j-z_k|,
\end{equation*}
which is conserved by the dynamics, can be used to control the minimum distance of vortices.
Moreover, \eqref{eq:condcollapse} is not compatible with $\xi_1+\xi_2+\xi_3=0$,
which means that one can not ``collapse three vortices into nothing'' as well as ``make three vortices burst out of nothing''.

A self-similar burst of three vortices starting from the origin $0\in\C$ is described explicitly by
\begin{equation}\label{eq:selfsimilar}
	w_j(t) = a_j Z(t), \quad
	Z(t) =	\sqrt{2at} \,e^{\imm \frac{b}{2a} \log t},\quad j=1,2,3,
\end{equation}
where $a_1,a_2,a_3\in\C$ and $a>0$, $b\in\R$ are numbers satisfying the relation
\begin{equation}\label{eq:asrelation}
\sum_{k \neq j}\frac{\xi_k}{a_j-a_k}=2\pi i \bar{a}_j (a-i b)
\end{equation}
(as one can verify substituting \eqref{eq:selfsimilar} into \eqref{eq:freevorticesC}).
For such solutions the centre of vorticity --another first integral\footnote{%
We thus have listed four first integrals: $I$, $H$ and the two components of $C$.
They are not in involution, and to integrate the system one must instead consider $I,H,|C|^2$.
We recall that already for $N=4$ vortices the system is in general not integrable.}
of the vortices system-- vanishes,
\begin{equation*}
	C=\sum_{j = 1}^N \xi_j w_j(t)\equiv 0,
\end{equation*}
coherently with the fact that three vortices are bursting out of one at $0\in\C$.
Since we will be interested in splitting a single vortex of arbitrary intensity $\xi\in\R \setminus \{0\}$ into three, let us stress the fact that this is always possible in the free case.

\begin{lemma}\label{lem:existselfsimilar}
	Let $\xi\in\R\setminus\set{0}$ and
	\begin{equation}\label{eq:intensities}
	\xi_1 = -\frac{1}{3} \xi, 
	\quad \xi_2=  \xi_3 =\frac{2}{3} \xi.
	\end{equation}
	Then there exist $a>0$, $b\in\R$, $a_1,a_2,a_3\in\C$ satisfying \eqref{eq:asrelation},
	so that $w=(w_1,w_2,w_3)$ given by \eqref{eq:selfsimilar} is a solution of \eqref{eq:freevorticesC} for $t>0$. Moreover, $w \in C^{1/2}((0,\infty),\C^3)$ and its H\"older seminorm $\bra{w_j}_{C^{1/2}}$, $j=1,2,3$ is controlled by
	\begin{align*}
	\bra{w_j}_{C^{1/2}} =|a_j| \bra{Z}_{C^{1/2}} \leq
	|a_j| \sqrt{2 a} \left(1+\frac{|b|}{4a}\right).
	\end{align*}
\end{lemma}

The proof is a straightforward computation we defer to \autoref{appendix}.
From now on, let $\xi\in\R\setminus\set{0}$ be fixed, and $a,b,a_1,a_2,a_3$
be a choice of parameters produced by \autoref{lem:existselfsimilar}.

We will prove in this section that one can obtain a burst of three vortices
satisfying the dynamics \eqref{eq:vortexternalfield}
including a suitable external vector field $f$,
as a perturbation of the self-similar, 
free solution described above.

\begin{prop}\label{prop:existenceexternal}
	For all $M,\rho>0$ there exist $T^*>0$ such that, for every $T<T^*$ and for every $f\in E_T$
	with $\|f\|_{E_T} \leq M$, there exists a solution $z:(0,T) \to \C^3\setminus\triangle^3$ of class $C^1$ of the equation
	\begin{equation}\label{eq:vortexternalfield}
	\dot{\bar{z}}_j =\frac{1}{2 \pi i} \sum_{k \neq j}\frac{\xi_k}{z_j - z_k}+f(t,z_j),
	\quad t \in (0,T),
	\quad j=1,2,3,
	\end{equation} 
	such that $\lim_{t \to 0} z_j(t) = 0$, $\bra{z_j}_{C^{1/2}_T} \leq 3 \bra{w_j}_{C^{1/2}}$ for every $j=1,2,3$ and 
	\begin{align} \label{eq:vortexternalfield_bc}
	\sup_{\substack{0<t < T}} |z(t)| \leq \rho.
	\end{align}
\end{prop}

We will then estimate the dependence of such solution on $f$ in \autoref{ssec:deponf}. 
Denote $\norm{f-g}_\infty=\sup_{0\leq t<T}\sup_{z\in\C}|f(t,z)-g(t,z)|$.
In particular, we prove there:

\begin{prop}\label{prop:dependenceonf_z}
	For all $M,\rho>0$ there exists $T^*>0$ such that, for every $T<T^*$, for every $f,g\in E_T$, 
	with $\norm{f}_{E_T}, \norm{g}_{E_T} \leq M$,
	and for any  $z , z'$ solutions to \eqref{eq:vortexternalfield} given by \autoref{prop:existenceexternal}, with external fields respectively $f$ and $g$,
	then it holds 
	\begin{align*}
	\sup_{\substack{0<t < T,\\j=1,2,3}} |z_j(t) - z_j'(t)| \leq C T^{1/2} \|f-g\|_{\infty},
	\end{align*}
	where $C>0$ is a constant depending on $\xi,a,b,a_1,a_2,a_3,M$.
\end{prop}

We remark that in \autoref{prop:dependenceonf_z} we establish continuity with respect to 
the external field \emph{only for the particular solution} to \eqref{eq:vortexternalfield} 
given by the construction of \autoref{prop:existenceexternal}. 
As a corollary, we also get uniqueness of solutions within the class of
curves remaining close to the self-similar solution, see \autoref{cor:uniqueness} below.
Finally, in \autoref{ssec:boundeddomain} we apply \autoref{prop:existenceexternal} to prove the existence of a burst of one vortex in three vortices in a bounded domain $D \subset \R^2$.

Our first step is to choose a coordinate system in which
our singular ODEs are more amenable for computations.

\subsection{A Convenient Coordinate System}

The characterisation of self-similar collapse of three free vortices of \cite{KrSt18}
was obtained by taking as coordinates the angles of the triangle formed by vortices positions.
Such a geometric approach does not fit well \eqref{eq:vortexternalfield},
so we adopt instead the coordinates used in \cite{oneil89}, which allow to
distinguish behaviours of solutions attaining to different time scales.

Let us introduce the variables $r, \theta \in \R$ , $x_2, x_3 \in \C$, defined by
\begin{equation*}
	(r e^{i \theta}, x_2, x_3)
	=\pa{\frac{z_1}{a_1}, \frac{z_2}{z_1} - \frac{a_2}{a_1}, \frac{z_3}{z_1} - \frac{a_3}{a_1}}.
\end{equation*}
The map $\Phi:(z_1,z_2,z_3) \mapsto (re^{i\theta},x_2,x_3)$ is a diffeomorphism of 
$(\C \setminus \{0\}) \times \C \times \C $ onto itself,
and the diagonal set $\triangle^3$ is given in terms of the new coordinates by
\begin{equation*}
\triangle^3 = \set{x_2 = 1 - \frac{a_2}{a_1} }\cup \set{x_3 = 1 - \frac{a_3}{a_1} }\cup \set{x_2 - x_3 = \frac{a_3 - a_2}{a_1}}.
\end{equation*}

The self-similar solution \eqref{eq:selfsimilar} of \eqref{eq:freevorticesC}
has a very simple form in this coordinates:
\begin{equation*}
	r(t)=\sqrt{2at},\quad \theta(t)=\frac{b}{2a}\log t \quad x_2(t)=x_3(t)\equiv 0.
\end{equation*}

\begin{rmk}
	To lighten notation, in what follows the couple of indices $(j,k)$ denotes
	any of the two couples $(2,3)$ and $(3,2)$.
	Indeed, in our choices index $1$ is distinguished, and will always be explicit. 
\end{rmk} 

The forthcoming lemma expresses the dynamics \eqref{eq:vortexternalfield}
in terms of $r,\theta,x_2,x_3$: doing so properly leads to somewhat complicated
expressions we leave to the \autoref{appendix}. 
Our concern is the singular part of the dynamics,
so we can leave unexpressed the more regular terms of the vector field 
--only reminding suitable bounds on them and their derivatives--
and focus on the leading order terms in the expansion at $t=0$.

\begin{lemma}\label{lem:changeofcoor}
	For any $\xi \neq 0$ there exists a choice of parameters $a,b,a_1,a_2,a_3$
	as in \autoref{lem:existselfsimilar}, and constants $C,\rho'>0$ depending only on $\xi, a_1,a_2,a_3$,
	such that for $r>0,\theta\in\R$ and $|x_2|, |x_3|, |x_2-x_3| < \rho'$,
	it holds $(re^{i\theta},x_2,x_3) \notin \triangle^3$ and the system \eqref{eq:vortexternalfield} is given in terms of $r,\theta,x_2,x_3$ by
	\begin{align*}
		\frac{d}{dt}(r^2)&=2a +\omega_r(x_2,x_3) +\frac{2}{|a_1^2|} \re \pa{z_1 f(t,z_1)},\\
		\frac{d}{dt}\theta&=\frac{b}{r^2} + \frac{\omega_\theta(x_2,x_3)}{r^2}+\im\pa{ \frac{\overline{f(t,z_1)}}{z_1} },\\
		\frac{d}{dt}x_j &=
		\frac{L_{j}(x_2,x_3,\overline{x_2},\overline{x_3})}{r^2} +
		\frac{\omega_j(x_2,x_3,x_2-x_3)}{r^2}
		+\frac{1}{z_1} \overline{f(t,z_j)}
		-\frac{z_j}{z_1^2}\overline{f(t,z_1)},	
	\end{align*}
	where:
	\begin{itemize}
		\item $\omega_r,\omega_\theta:\set{z\in\C:|z|< \rho'}^2\to\C$ are holomorphic functions such that 
		\begin{equation*}
		|\omega_{r,\theta}(x_2,x_3)| \leq C \left(|x_2| + |x_3|\right),\quad
		|\nabla \omega_{r,\theta}(x_2,x_3)| \leq C;
		\end{equation*}
		\item $L_2,L_3:\C^4\to\C$ are $\C$-linear functions,
		and the $4\times 4$ matrix $(L_2,L_3,\bar L_2,\bar L_3)$
		has eigenvalues whose real part is equal to $-a$;
		\item $\omega_2,\omega_3:\set{z\in\C:|z|< \rho'}^3\to\C$ are holomorphic functions such that
		\begin{gather*}
		|\omega_{j}(x_2,x_3,x_2-x_3)| \leq C \left(|x_2|^2 + |x_3|^2\right),\\
		|\nabla \omega_{j}(x_2,x_3,x_2-x_3)|\leq C
		\left(|x_2| + |x_3|\right).
		\end{gather*} 
	\end{itemize}
\end{lemma}

\noindent
Notice that $z_1,z_2,z_3$ appearing on the right-hand side of equations above can be expressed explicitly in
terms of $r,\theta,x_2,x_3$ by means of $\Phi^{-1}$, but we refrain from doing this to keep the notation as simple as possible.
The proof of \autoref{lem:changeofcoor} consists in a lengthy but straightforward computation
based on the expansion $(\alpha-z)^{-1}=\alpha^{-1}(1+z/\alpha+z^2/\alpha^2+\dots)$
and the relations between parameters $\xi,a,b,a_j$, and we thus defer its proof to \autoref{appendix}.

\begin{rmk}
	Let us emphasise the fact that in the latter result we claim that \emph{there exists at least one}
	choice of parameters $a,b,a_1,a_2,a_3$ for which \eqref{eq:selfsimilar} solves
	\eqref{eq:freevorticesC} \emph{that also satisfies the other parts of the statement}.
	In fact, the condition on eigenvalues of $(L_2,L_3,\bar L_2,\bar L_3)$
	is \emph{not} true for all self-similar solutions, and one can exhibit
	numerical examples in which the matrix has eigenvalues with positive real part. 
\end{rmk}

\subsection{A Fixed Point Problem}

We can now begin to discuss the proof of \autoref{prop:existenceexternal}.
Let $a_1,a_2,a_3$ and $a,b$ be as in \autoref{lem:changeofcoor}:
besides those coefficients, the datum of the problem is the external field
$f$, on which we can make right away a convenient reduction.

\begin{rmk}\label{lem:wlogfzero}
	In proving \autoref{prop:existenceexternal}, we can assume without loss of generality
	that $f(0,0)=0$.
	Indeed, let $f(0,0) = A \in \C$ and denote $\tilde{z}_j(t) = z_j(t) - \bar{A}t$, $j=1,2,3$. Then, $(z_1,z_2,z_3)$ solves \eqref{eq:vortexternalfield}
	if and only if $( \tilde z_1,\tilde z_2,\tilde z_3)$ solves the same system with $f$ replaced by
	\begin{equation*}
	\tilde f(t,p)= f(t,p +\bar{A} t)- A,
	\quad t \in [0,T), \quad p \in \C,
	\end{equation*}
	which satisfies $\tilde{f} \in E_T$ and $\|{\tilde{f}}\|_{E_T} \leq 2 \norm{f}_{E_T}$, and therefore the assumptions of \autoref{prop:existenceexternal} still hold for $\tilde{f}$ up to the choice of a possibly smaller $M$.
\end{rmk}

In order to prove \autoref{prop:existenceexternal}, in the first place we look for a solution of the system in \autoref{lem:changeofcoor},
and we will find it as a fixed point of a continuous map. 
In particular, we will make use of Schauder fixed point Theorem \cite[Corollary 2.13]{Ze86}, 
stating that every continuous map from a non-empty, compact, 
convex subset of a Banach space to itself admits at least a fixed point. 

We begin by defining the space in which to set the fixed point problem:
in doing so, we will consider as two distinct variables the modulus and angle of $re^{i\theta}$,
for which we introduce new symbols to avoid confusion between function spaces,
see \autoref{rmk:fixedpoint}.
For $T>0$, to be chosen small enough in what follows, denote
\begin{align}\nonumber
	\U_T &:= \Big \{  (\zeta,\eta,x_2,x_3) \in C((0,T),\R^2 \times \C^2): \\
	\label{eq:bc}
	&\qquad \lim_{t \to T } \eta(t) = 0;\\
	\label{eq:distance}
	&\qquad |\zeta(t) - 2at| \leq t^{3/2}, 
	\,|x_2(t)|,|x_3(t)| \leq t \quad \forall t \in (0,T);\\
	\label{eq:holder}
	&\qquad \bra{\zeta}_{C^{1/2}_T},
	\bra{\eta}_{C^{1/2}_T}, \bra{x_2}_{C^{1/2}_T}, \bra{x_3}_{C^{1/2}_T} \leq 1 \Big\}.
\end{align}
Variables $\zeta$ and $\eta$ are respectively aliases 
of $\zeta(t) = r^2(t)$ and $\eta(t) = \theta(t) - \frac{b}{2a} \log t$. 
We set the fixed point argument with respect to these variables because:
\begin{lemma}
	The set $\U_T\subset C((0,T),\R^2 \times \C^2)$ is convex
	and compact with respect to the uniform topology induced by the norm $\norm{\cdot}_\infty$.
\end{lemma}
\noindent
Indeed, compactness follows by Ascoli-Arzelà Theorem,
and checking convexity is straightforward.

Let us stress the fact that the need of convexity is what forces us to introduce $\zeta,\eta$: 
we will use them \emph{only} in the present paragraph,
and revert to $re^{i\theta},x_2,x_3$ from \autoref{ssec:deponf} on,
see \autoref{rmk:fixedpoint} below.
   
Let us briefly comment on properties \eqref{eq:bc}-\eqref{eq:holder}.
First of all, \eqref{eq:distance} implies 
\begin{align*}
\lim_{t \to 0 } (\zeta(t),x_1(t),x_2(t)) = (0,0,0),
\end{align*} 
encoding the fact that vortices burst out of the origin and thus imposing a boundary condition on these variables;
boundary condition on $\eta$ (or equivalently, on $\theta(t) = \eta(t) + \frac{b}{2a} \log t$) is given at final time $T$ by \eqref{eq:bc}.

Conditions \eqref{eq:distance}, together with the control on $\eta$, impose that elements of $\U_T$ be close to the self-similar solution $w$,
the latter belonging to $\U_T$ for every $T>0$ by \autoref{lem:existselfsimilar},
and in coordinates $\zeta,\eta,x_2,x_3$ the self similar solution is given by $\zeta(t) = 2at$, $\eta(t)=x_2(t)=x_3(t)=0$. 
The first of conditions \eqref{eq:distance} 
ensures $\zeta(t)>0$ for every $t \in (0,T)$ and $T$ small enough, 
coherently with the interpretation of $\zeta$ as the variable $r^2$ in \autoref{lem:changeofcoor}.
Finally, \eqref{eq:holder} impose H\"older regularity to provide compactness, as discussed above.

Let us assume $u=(\zeta,\eta,x_2,x_3)\in\mathcal{U}_{T}$
and rewrite equations of \autoref{lem:changeofcoor} in integrated form,
in order to set up a fixed point problem in $\U_T$.
To further reduce notation, we will write $\x=(x_2,x_3,\overline{x_2},\overline{x_3})$ and
$L(\x)=(L_2,L_3,\bar L_2,\bar L_3)(\x)$, thus coupling equations for $x_2,x_3$
with their complex conjugates. We will denote projection on coordinates by $x_j=P_j[\x]$.
Finally, in what follows $C>0$ denotes a constant depending on $\xi,a,b,a_1,a_2,a_3$ 
and $\norm{f}_{E_T}$, possibly differing in every occurrence.

\begin{lemma}\label{lem:U}
	Consider a function $u=(\zeta,\eta,x_2,x_3)\in\mathcal{U}_{T}$ satisfying
	\begin{align}
	\nonumber
	\zeta(t) &= 2at + \int_0^t R(s,u_s) ds, \\
	\nonumber
	\eta(t)&= \int_{T}^t \Theta(s,u_s) ds,\\
	\label{eq:xequation}
	\x(t) &= \int_0^t \frac{L(\x_s)}{2as} ds + \int_0^t \Xi(s,u_s) ds,
	\end{align}
	where 
	\begin{align*}
		R(s,u) &=\omega_{r}(x_2,x_3) +\frac{2}{|a_1^2|} \re\pa{z_1 f(s,z_1)},\\
		\Theta(s,u) &= \left(\frac{b}{\zeta} -\frac{b}{2as} \right)
		+ \frac{\omega_{\theta}(x_2,x_3)}{\zeta}
		+\im\pa{ \frac{\overline{f(s,z_1)}}{z_1} },\\
		\Xi&=(\Xi_2,\Xi_3,\bar\Xi_2,\bar\Xi_3),\\
		\Xi_j(s,u) &= L_j(\x)\pa{\frac1{\zeta}-\frac1{2as}} +
		\frac{\omega_j(x_2,x_3,x_2-x_3)}{\zeta}
		+\frac{1}{z_1} \overline{f(s,z_j)}
		-\frac{z_j}{z_1^2}\overline{f(s,z_1)}.
	\end{align*}
		
	Then $(r,\theta,x_2,x_3)=(\sqrt\zeta,\eta+\frac{b}{2a}\log t,x_2,x_3)$ solve
	the system of \autoref{lem:changeofcoor},
	and we have the following \emph{a priori} estimates for every $u,u' \in \U_T$
	solving the above system:
	\begin{itemize}
		\item $|R(s,u_s)| \leq C s$ and 
		\begin{align*}
		|R(s,u_s) - R(s,u'_s)| \leq C \left(|\zeta_s - \zeta'_s|^{1/2} + |\eta_s - \eta'_s|+ |\x_s - \x'_s| \right);
		\end{align*}	
		\item making use of $f(0,0)=0$ and the estimates on $R$, $|\Theta(s,u_s)|\leq C$ and 
		\begin{align*}
		|\Theta(s,u_s) - \Theta(s,u'_s)| \leq \frac{C}{s} \left(|\zeta_s - \zeta'_s|^{1/2} 
		+ |\eta_s - \eta'_s|+ |\x_s - \x'_s| \right);
		\end{align*}	
		\item making again use of $f(0,0)=0$ and the estimates on $R$,
		$|\Xi(s,u_s)| \leq C s^{1/2}$ and 
		\begin{align*}
		|\Xi(s,u_s) - \Xi(s,u'_s)| \leq C 
		\left(|\zeta_s - \zeta'_s|^{1/2} + |\eta_s - \eta'_s|+ |\x_s - \x'_s| \right).
		\end{align*}	
	\end{itemize}
\end{lemma}

%
%

Let us focus now on the equation for $\x$: it is indeed the most delicate to treat, 
since we have to check that condition \eqref{eq:distance} holds for its solutions, 
notwithstanding the fact that \eqref{eq:distance} is not given by trivial a priori bounds on the right-hand side of \eqref{eq:xequation}.
Applying the method of variation of constants, and looking for solutions of the form
\begin{equation*}
\x(t) = \exp\left( \frac{\log t}{2a} L \right) c(t),
\end{equation*} 
where we denote by $\exp$ the matrix exponential, we obtain
\begin{equation*}
\x(t) = \exp\left( \frac{\log t}{2a} L \right)\int_{t_0}^{t}
\exp\left( -\frac{\log s}{2a} L \right)
\Xi(s,u_s) ds.
\end{equation*}
This still is (part of) a differential equation in the unknowns $\zeta,\eta,\x$ 
since the value of the latter quantities at time $s$ is contained in the term $\Xi(s,u_s)$.
By \autoref{lem:changeofcoor}, all eingenvalues of $L$ have real part equal to $-a<0$.
Imposing $\x(0)=0$ we deduce that $t_0$ must be equal to zero. 
Moreover, by \autoref{lem:changeofcoor} and the above estimates on $\Xi$, the integral converges in the $0$ extreme since
\begin{equation*}
\abs{\exp\pa{-\frac{\log s}{2a} L}\Xi(s,u_s)}\leq C s.
\end{equation*}

With this at hand, we are ready to prove:
\begin{lemma} \label{lem:contGamma}
	There exists $T^*>0$ sufficiently small such that, 
	for any $T<T^*$ and $u=(\zeta,\eta,x_2,x_3)\in \U_{T}$,
	the expression
	\begin{align*}
		\Gamma(u)&=\Gamma(\zeta,\eta, x_2, x_3) = (\tilde{\zeta},\tilde{\eta}, \tilde{x}_2, \tilde{x}_3),\\
		\tilde{\zeta}(t) &= 2at + \int_0^t R(s,u_s) ds, \\
		\tilde{\eta}(t) &= \int_{T}^t \Theta(s,u_s) ds,\\
		\tilde{x}_j(t) &= P_j \bra{\exp\left( \frac{\log t}{2a} L \right)\int_{t_0}^{t}
		\exp\left( -\frac{\log s}{2a} L \right)
		\Xi(s,u_s) ds},
	\end{align*}
	defines a map $\Gamma:\U_{T} \to \U_{T}$
	that is continuous in the uniform topology induced by $C\pa{(0,T),\R^2 \times \C^2} \supset \U_{T}$.
\end{lemma}

\begin{rmk} \label{rmk:fixedpoint}
	By \autoref{lem:U}, given a fixed point $u = (\zeta,\eta,x_2,x_3) \in\U_T$, $\Gamma(u)=u$,
	we can produce a solution $z \in C((0,T),\C^3)$ of 
	\eqref{eq:vortexternalfield} setting
	\begin{gather*}
		\Psi: \U_T \to C((0,T),\C^3),\\ \Psi(\zeta,\eta,x_2,x_3) = (r e^{i \theta},x_2,x_3),
		\quad r^2(t) = \zeta(t), \, \theta(t) = \eta(t) + \frac{b}{2a} \log t,
	\end{gather*}
	and taking $z(t) = \Phi^{-1}(\Psi(u)_t)$.
	We will denote from now on $\XX_T= \Psi(\U_T)$. 
\end{rmk}

\begin{proof}
	It is immediate to observe that $\Gamma$ takes values in $C\pa{(0,T),\R^2 \times \C^2}$.
	
	In the first place, we prove that $\Gamma$ actually takes values in $\U_T$.
	Condition \eqref{eq:bc} on $\tilde{\eta}$ immediately follows from $|\Theta(s,u_s)| \leq C$.
	Then, since for $T$ small enough and $t<T$ 
	\begin{align*}
	|\tilde{\zeta}(t)-	2at| 
	&\leq \left| \int_0^t R(s,u_s) ds \right|
	\leq Ct^2 \leq t^{3/2},
	\\
	|\tilde{x}_j(t)| &\leq
	\int_0^t \left\|
	\exp\left( \frac{\log t -\log s}{2a} L \right)
	\right\|
	|\Xi(s,u_s)| ds
	\\
	&\leq C \int_0^t e^{-\frac{\log t - \log s}{2} } s^{1/2} ds	\leq C t^{3/2} \leq t,
	\end{align*}
	by the estimates established in \autoref{lem:U}, conditions \eqref{eq:distance} hold for $\tilde{\zeta}$ and $\tilde{x_j}$, $j=2,3$. 
	Therefore, in order to finish the proof that $\Gamma$ takes values in 
	$\U_T$, we only need to check conditions \eqref{eq:holder}.
	These follow again from \autoref{lem:U} and
	\begin{align*}
	|\tilde\zeta(t')-	\tilde\zeta(t)| 
	&\leq 2a|t'-t| + \left| \int_t^{t'} R(s,u_s) ds \right| \leq |t'-t|^{1/2},\\
	|\tilde{\eta}(t')-\tilde{\eta}(t)| 
	&\leq \left| \int_t^{t'} \Theta(s,u_s) ds \right| \leq C|t'-t| \leq |t'-t|^{1/2},	\\	
	|\tilde{x}_j(t') - \tilde{x}_j(t)| 
	&\leq \abs{\int_t^{t'} \dot{\tilde{x}}_j(s) ds}
	\leq \left|\int_t^{t'} \left( \frac{L(\tilde{\mathbf{x}})}{2as} + \Xi(s,u_s) \right) ds \right|	\\
	&\leq C |t'-t| \leq |t'-t|^{1/2},
	\end{align*}
	possibly requiring a further restriction on $T$. 
	This gives \eqref{eq:holder} for $\Gamma(u)$ and concludes the proof that $\Gamma$ maps $\U_{T}$ into itself.

	Moving on to continuity, let $u,u' \in \U_{T}$ 
	and denote $\tilde u=\Gamma(u)$, $\tilde u'=\Gamma(u')$.
	The term $\tilde{\zeta}$ is the easiest one:
	\begin{align*}
	|\tilde{\zeta}(t) - \tilde{\zeta}'(t)| \leq \int_0^t |R(s,u_s)-R(s,u'_s)| ds
	\leq C t \left( \|u - u'\|_\infty + \|u - u'\|^{1/2}_\infty \right).
	\end{align*}
	Consider now $\tilde\eta,\tilde\eta'$. Fix $\delta \in (0,T)$: for $t \in (\delta,T)$ we have
	\begin{align*}
	|\tilde{\eta} (t) - \tilde{\eta}' (t)|
	\leq&
	\int_\delta^T
	\left| \Theta(s,u_s)-\Theta(s,u'_s)\right| ds
	\leq
	C\int_\delta^T
	\frac{|u_s-u'_s|+|u_s-u'_s|^{1/2}}{s}ds
	\\
	\leq&
	C \left( \|u - u'\|_\infty +\|u - u'\|^{1/2}_\infty  \right)
	\left( \log T -\log \delta \right),
	\end{align*}
	while for $t \in (0,\delta)$ 
	\begin{align*}
	|\tilde{\eta} (t) - \tilde{\eta}' (t)|
	\leq&
	\int_\delta^T
	\left| \Theta(s,u_s)-\Theta(s,u'_s)\right| ds
	+
	\int_0^\delta
	\left| \Theta(s,u_s)-\Theta(s,u'_s)\right| ds
	\\
	\leq&
	C \left( \|u - u'\|_\infty +\|u - u'\|^{1/2}_\infty  \right)
	\left( \log T -\log \delta \right)
	+ C \delta.
	\end{align*}
	Therefore, for every $\delta \in (0,T)$
	\begin{align*}
	\|\tilde{\eta} - \tilde{\eta}'\|_\infty \leq
	C \left( \|u - u'\|_\infty +\|u - u'\|^{1/2}_\infty  \right)
	\left( \log T -\log \delta \right)
	+ C \delta.
	\end{align*}
	Thus $\tilde \zeta$ and $\tilde \eta$ are continuous functions of $u$, and we are left to control components $\tilde x_j$.	
	By the usual estimates on $\Xi$ we have
	\begin{multline*}
	|\tilde{x}_j (t) - \tilde{x}'_j (t)| 
	\leq
	\left| 
	\int_0^t
	\exp\left( \frac{\log t -\log s}{2a} L \right)
	\left( \Xi(s,u_s)-\Xi(s,u'_s) \right) ds
	\right|
	\\
	\leq
	t^{-\frac{1}{2}}
	\int_0^t
	s^{\frac{1}{2}}
	\left| \Xi(s,u_s)-\Xi(s,u'_s) \right| ds
	\leq
	C t \left( \|u - u'\|_\infty +\|u - u'\|^{1/2}_\infty  \right),
	\end{multline*}
	which concludes the proof.
\end{proof}

\begin{lemma} \label{lem:holder_fixed}
There exists $T^*>0$ sufficiently small such that, for any $T<T^*$, $u \in \U_T$ and 
$\tilde{u} = (\tilde{\zeta},\tilde{\eta},\tilde{x}_2,\tilde{x}_3) = \Gamma(u)$, 
denoting $r^2(t) = \tilde{\zeta}(t)$ and $\theta(t) = \tilde{\eta}(t) + \frac{b}{2a} \log t$, 
then $r e^{i \theta} \in C^{1/2}((0,T),\C)$ and $\bra{r e^{i \theta}}_{C^{1/2}_T} \leq 2 \bra{Z}_{C^{1/2}}$.
\end{lemma}	

\begin{proof}
	Throughout the proof we take $T$ smaller than the value of $T^*$ given by \autoref{lem:contGamma}.
	We need to evaluate first, for any $t < t'$
	\begin{align*}
	|r(t')-r(t)|
	&\leq \sqrt{|r^2(t')-r^2(t)|}\leq (2a+CT)^{1/2}|t'-t|^{1/2},	\\	
	\left|{\theta}(t')-{\theta}(t)  \right|
	&\leq \abs{\frac{b}{2a} \int_t^{t'} \frac{ds}{s} +\int_t^{t'}  \Theta(s,u_s) ds}\\
	&\leq \abs{\frac{b}{2a \sqrt{t}} \int_t^{t'} \frac{ds}{\sqrt{s}}} +
	C |t'-t| \leq \left( \frac{|b|}{4a} \frac{1}{\sqrt{t}} + CT^{1/2} \right) |t'-t|^{1/2}.
	\end{align*}
	In addition, we have the following estimate
	\begin{align*}
	|r(t)-  \sqrt{2at}| =
	\left|
	\frac{r^2(t) - 2at}{r(t) + \sqrt{2at}}  \right| 
	\leq \frac{CT^{3/2}}{\sqrt{2a}},
	\end{align*}
	and thus
	\begin{align*}
	|r(t') e^{i \theta(t')}-  r(t) e^{i \theta(t)}|
	\leq&\,
	|r(t') e^{i \theta(t')}-  r(t) e^{i \theta(t')}|
	+
	|r(t) e^{i \theta(t')}-  r(t) e^{i \theta(t)}|
	\\
	\leq&\,
	|r(t') -  r(t)|
	+
	r(t)| e^{i \theta(t')} -   e^{i \theta(t)}|
	\\
	\leq&\,
	|r(t') -  r(t)|+\left( \sqrt{2at} + \frac{C T^{3/2}}{\sqrt{2a}} \right) 
	\left|{\theta}(t')  -{\theta}(t)\right|\\
	\leq& (2a+CT)^{1/2}|t'-t|^{1/2}\\
	&+ \left( \sqrt{2at} + \frac{C T^{3/2}}{\sqrt{2a}} \right) 
	\pa{\frac{|b|}{4a} \frac{1}{\sqrt{t}} + CT^{1/2}} |t'-t|^{1/2}\\
	\leq& 2 \bra{Z}_{C^{1/2}},
	\end{align*}
	taking, in the last line, $T$ sufficiently small.
\end{proof}	

Now we have all the ingredients for:
\begin{proof}[Proof of \autoref{prop:existenceexternal}]
By \autoref{lem:contGamma} above, $\Gamma$ is a continuous function from 
$\U_{T}$ to itself for every $T>0$ sufficiently small. 
Since $\U_{T}$ is a compact, convex subset of the Banach space $C((0,T),\R^2 \times \C^2)$, 
and it is non-empty by \autoref{lem:existselfsimilar}, 
Schauder fixed point Theorem gives the existence of a fixed point 
$\Gamma(u)=u \in \U_{T}$, 
and by \autoref{rmk:fixedpoint} this gives the existence of a 
solution $z$ to \eqref{eq:vortexternalfield} satisfying $\lim_{t \to 0} z_j(t) = 0$ for every $j=1,2,3$. 

The H\"older seminorm $\bra{z_j}_{C^{1/2}_T}$ is controlled for small $T$ using $z_1 = a_1 r e^{i \theta} $ and $z_j = r e^{i \theta} ( a_1 x_j + a_j)$ for $j \neq 1$, and the bounds $\bra{r e^{i\theta}}_{C^{1/2}_T} \leq 2 \bra{Z}_{C^{1/2}}$ given by \autoref{lem:holder_fixed}, and $|x_j(t)| \leq t$ valid for every $u \in \U_{T}$.  
 
Finally, to check \eqref{eq:vortexternalfield_bc} one can simply observe 
that $u \in \U_{T}$ and the formulas above imply
\begin{align*}
\sup_{\substack{t < T}} |z(t)| \leq \left(\sqrt{2aT}
+ \frac{CT^{3/2}}{\sqrt{2a}}\right) \left( |a_1|+|a_2|+|a_3|\right) \left(1+T \right),
\end{align*}
so the thesis is complete taking $T$ small enough.
\end{proof}

\subsection{Dependence on the external field}\label{ssec:deponf}
Once established the {existence} of three-vortices bursts in an external field,
it is natural to ask how these singular solution depend on the latter.
Indeed, not only this leads to discuss \emph{uniqueness} of solutions,
but it is a fundamental step in generalising our strategy to bursts of three
vortices in a larger system of vortices.

Recall the definition of $\XX_T$ given in \autoref{rmk:fixedpoint}.
Taking advantage of the explicit form of $\Phi^{-1}$, \autoref{prop:dependenceonf_z} follows immediately from the following:

\begin{prop}\label{prop:dependenceonf}
	For all $M,\rho>0$ there exists $T^*>0$ such that, for every $T<T^*$, for every $f,g\in E_T$, 
	with $\norm{f}_{E_T}, \norm{g}_{E_T} \leq M$,
	and for any  $x , x' \in \XX_T$ solutions to \eqref{eq:vortexternalfield} given by the construction of \autoref{prop:existenceexternal} and \autoref{rmk:fixedpoint}, with external fields respectively $f$ and $g$,
	then it holds 
	\begin{align*}
	\sup_{0<t < T } |x(t) - x'(t)| \leq C T^{1/2} \|f-g\|_{\infty} 
	\end{align*}
	where $C>0$ is a constant depending on $\xi,a,b,a_1,a_2,a_3,M$.
\end{prop}

\begin{proof}
	Let us write the solutions relative to $f,g$ as
	\begin{equation*}
		x=(r e^{i\theta},x_2,x_3), \quad x'=(r' e^{i\theta'},x_2',x_3'),
	\end{equation*}
	and use the vortices equations in the form of \autoref{lem:contGamma},
	writing $R_f$, $\Theta_f$ and $\Xi_f$
	in the equations for $x$ and $R_g,\Theta_g,\Xi_g$ in the ones for $x'$ (with a little abuse of notation due to the fact that here we work with the space $\XX_T$ instead of $\U_T$).
	
	In what follows $C$ denotes (possibly different) positive constants depending on
	parameters $\xi,a,b,a_1,a_2,a_3$ and $M$. Using the key regularity assumption $f,g \in E_T$ and arguing as in \autoref{lem:U}, we have the following estimates:
	\begin{align*}
	|R_f(s,x_s) - R_g(s,x'_s)| \leq&
	|\omega_{r}(x_2,x_3)-\omega_{r}(x'_2,x'_3)| 
	\\
	&+
	\frac{2 a_1}{|a_1^2|} 
	\left| x_1 f(s,a_1 x_1) - x'_1 g(s,a_1 x'_1)\right|
	\\
	\leq& 
	C |x_s-x'_s| + C s^{1/2} \|f-g\|_{\infty},\\
	|\Theta_f(s,x_s) - \Theta_g(s,x_s)| \leq&
	\left|\frac{b}{r^2} -\frac{b}{r'^2}\right| +
	\left|\frac{\omega_{\theta}(x_2,x_3)}{r^2}-\frac{\omega_{\theta}(x'_2,x'_3)}{r'^2} \right|
	\\
	&+
	\left| \frac{f(s,a_1 x_1)}{x_1} - \frac{ g(s,a_1 x'_1)}{x'_1}\right|
	\\
	\leq& 
	C s^{-1} |x_s-x'_s| + Cs^{-1/2} \|f-g\|_{\infty},\\
	|\Xi_f(s,x_s) - \Xi_g(s,x'_s)| \leq&
	C |x_s-x'_s| + C s^{-1/2} \|f-g\|_{\infty}.
	\end{align*}
	
	With that being said, let us move to the actual proof. We can estimate
	\begin{align*}
	|x_j(t) - x'_j(t)| &\leq \int_{0}^{t} |\Xi_f(s,x_s) - \Xi_g(s,x'_s)| ds\\
	&\leq C\int_{0}^{t} |x_s-x'_s| ds 
	+ C T^{1/2} \|f-g\|_{\infty},\\
	|r(t) - r'(t)| &\leq \frac{|r^2(t) - r'^2(t)|}{|r(t) + r'(t)|} 
	\leq \frac{C}{\sqrt{t}} \int_{0}^{t} |R_f(s,x_s) - R_g(s,x'_s)| ds \\
	&\leq \frac{C}{\sqrt{t}}
	\int_{0}^{t} |x_s-x'_s| ds +C T \|f-g\|_{\infty},\\
	|\theta(t) - \theta'(t)| &\leq \int_{t}^{T} |\Theta_f(s,x_s) - \Theta_g(s,x'_s)| ds\\
	&\leq C \int_{t}^{T} \frac{|x_s-x'_s|}{s} ds
	+C T^{1/2} \|f-g\|_{\infty}.
	\end{align*}
	Therefore, proceeding as in the proof of \autoref{prop:existenceexternal}:
	\begin{align*}
	\left| r(t) e^{i \theta(t)} - r'(t)e^{i \theta'(t)} \right| 
	\leq& 	|r(t) - r'(t)| + |r'(t)| \left|e^{i \theta(t)} - e^{i \theta'(t)} \right| \\
	\leq& 	C T \|f-g\|_{\infty} +C \int_0^T \frac{|x_s-x'_s|}{\sqrt{s}}ds.
	\end{align*}
	Thus, for every $t \in (0,T)$
	\begin{align*}
	|x_t - x'_t | &\leq C T^{1/2} \|f-g\|_{\infty} +
	C \int_0^T \frac{|x_s-x'_s|}{\sqrt{s}}ds
	\\
	&\leq C T^{1/2} \|f-g\|_{\infty} +
	C T^{1/2} \|x-x'\|_\infty.
	\end{align*}
	Taking $T>0$ small enough so that $CT^{1/2}<1$, we obtain the thesis.	
\end{proof}

As a corollary of the previous result, we obtain uniqueness of solutions 
of class $\XX_T$ up to a certain time $T^*$ 
(possibly smaller than the existence time of solutions given by \autoref{prop:existenceexternal}).

\begin{cor} \label{cor:uniqueness}
	For all $M,\rho>0$ there exist $T^*>0$ such that, for every $T<T^*$ and for every $f\in E_T$
	with $\|f\|_{E_T} \leq M$, there exists a unique solution $z$ of \eqref{eq:vortexternalfield}, satisfying \eqref{eq:vortexternalfield_bc} and $\Phi(z) \in \XX_T$.
\end{cor}

\subsection{Burst of Three Vortices in a Bounded Domain} \label{ssec:boundeddomain}
Before moving to systems of $N$ vortices on $\C$,
we show in this paragraph how the main result of this section can be used to 
prove the existence of bursting vortices in bounded domains.

Let $D \subseteq \R^2$ be an open, bounded, connected and simply connected 
domain with smooth boundary.
The Green function $G_D$ associated to the Laplace operator in the domain 
$D$ with zero Dirichlet boundary condition on $\partial D$ has the form:
\begin{align*}
G_D(x,y) = G(x,y) + \gamma_D(x,y),
\end{align*}
where $G(x,y)=- \frac{1}{2\pi} \log |x-y|$ is the Green function on $\R^2$ and 
$\gamma_D(x,\cdot)$ is given by the harmonic continuation of $G(x,\cdot)$ from $\partial D$ to the whole of $D$:
\begin{align*}
\Delta_y \gamma_D(x,y) = 0, &\quad y \in D, \\
\gamma_D(x,y) = -G(x,y), &\quad y \in \partial D,
\end{align*}
where $\Delta_y$ denotes the Laplace operator in the $y$ variable;
$\gamma = \gamma_D$ is a smooth function of both its variables and is harmonic in its second variable.
 
In this setting, Euler equations for point vortices $z_j$, $j=1,2,3$ are given by
	\begin{equation}\label{eq:vortexboundeddomain}
	\dot{\bar{z}}_j =\frac{1}{2 \pi i} \sum_{k \neq j}\frac{\xi_k}{z_j - z_k}
	-\sum_{k=1}^3 \xi_k \overline{\nabla^{\perp}_{x} \gamma(z_j,z_k)},
	\quad t \in (0,T),
	\quad j=1,2,3.
	\end{equation} 
Just as in the free case, the system is defined coherently with Euler's equations:
we refer to \cite[Section 4.1]{MaPu94} for a proper introduction to the vortices on bounded domains.
In this section we prove the following:
\begin{prop}\label{prop:existenceboundeddomain}
	Take $z_0 \in D$. Then there exist $T>0$ and a solution $z:(0,T) \to D^3\setminus\triangle^3$ of \eqref{eq:vortexboundeddomain}	such that $\lim_{t \to 0} z_j(t) = z_0$ for every $j=1,2,3$.
\end{prop}

To simplify notation, we will take $z_0 = 0\in D$;
this is without loss of generality, since translating the domain does not change the problem.
Our strategy is to interpret the boundary term in \eqref{eq:vortexboundeddomain} 
as an external field $f$ acting on a free system of vortices:
\begin{align*}
f(t,p) = F(z_1(t),z_2(t),z_3(t);p) =-\sum_{k=1}^3 \xi_k \overline{\nabla^{\perp}_{x} \gamma(p,z_k(t))}, 
\quad p \in \C,
\end{align*}
and then apply \autoref{prop:existenceexternal}. 
Indeed, as long as the points $z_j$ are far enough from the boundary $\partial D$, 
there is no problem in considering $z$ as taking values in $\C^3\setminus\triangle^3$ 
instead of $D^3\setminus\triangle$.
However, the external field $f$ acting on the point vortices 
now depends on the positions $z_j$ of \emph{all} vortices, unlike in \autoref{prop:existenceexternal}: 
to deal with this issue, we will use the following fundamental observation. 

Define $\rho = d(0,\partial D)/2$ and
denote $B_\rho = \{ p \in \C : |p| \leq \rho \} \subset D \subset \C$. 
Let $z \in C((0,T),B^3_\rho)$ be a candidate solution of \eqref{eq:vortexboundeddomain},
and consider the centre of vorticity
\begin{align*}
C_z(t) = \sum_{j=1}^3 \xi_j z_j(t);
\end{align*}
one can compute the derivative of $C_z$ as a function of time, and notice that 
\begin{align*}
\left|\frac{d}{dt} C_z(t) \right| =
\left|\sum_{j=1}^3 \xi_j \dot{z}_j(t)\right|
&=
\left|\frac{i}{2\pi}\sum_{j=1}^3
\sum_{k \neq j} \frac{\xi_j \xi_k}{\bar{z}_j-\bar{z}_k} -
\sum_{j,k=1}^3 \xi_j \xi_k {\nabla^{\perp}_{x} \gamma(z_j,z_k)}\right|
\\
&=
\left|\sum_{j,k=1}^3 \xi_j \xi_k {\nabla^{\perp}_{x} \gamma(z_j,z_k)}\right|
\leq c,
\end{align*}
where $c$ is a constant depending only on $\xi_j$, $j=1,2,3$ and $\rho$. Therefore, 
the map $C_z:(0,T) \to D$ is Lipschitz continuous,
and since $\lim_{t \to 0} z_j(t) = 0$ for every $j$, we also deduce $C_z(t) \leq ct$ for every $t \in (0,T)$.

\begin{proof}[Proof of \autoref{prop:existenceboundeddomain}]
Let $\rho, c$ be as above and $-2\xi_1 = \xi_2 = \xi_3 \neq 0$. Define
\begin{align*}
\mathcal{Z}_T =& 
\left\{ z \in C^{1/2}((0,T),B^3_\rho) :\right.
\\
& \qquad \lim_{t \to 0 } z_j(t) = 0, 
\, \bra{z_j}_{C^{1/2}_T} \leq 3 \bra{w_j}_{C^{1/2}} 
\, \forall j=1,2,3; 
\\
&\left. \qquad
 C_z \in \lip((0,T),D), \,
\bra{C_z}_{\lip} \leq c \,\right\}.
\end{align*}
For every $z \in \mathcal{Z}_T$ and $p \in B_\rho$, one has
\begin{align*}
\sum_{k=1}^3 \xi_k {\nabla^{\perp}_{x} \gamma(p,z_k)}
-
\nabla^{\perp}_{x} \gamma(p,0) -
\nabla_y \nabla^{\perp}_{x} \gamma(p,0) \cdot C_z(t) = \omega_\gamma(p,z(t)),
\end{align*}
where $\omega_\gamma$ is a smooth function satisfying for every $p \in B_\rho$ and $q \in B^3_\rho$
\begin{align*}
\left| \omega_\gamma(p,q)-
\omega_\gamma(p,q') \right|
\leq
C |q-q'|^2
\end{align*}
for some constant $C$ depending only on $\xi_j$, $j=1,2,3$ and $\rho$.
As a consequence, the function
\begin{align*}
f(t,p) = F(z_1(t),z_2(t),z_3(t);p) = -\sum_{k=1}^3 \xi_\ell \overline{\nabla^{\perp}_{z} \gamma(p,z_k(t))},
\quad t \in [0,T), 	\quad p \in B_\rho,
\end{align*}
is in $E_{\rho,T} = C([0,T), C^2(B_\rho,\C)) \cap\lip ([0,T), C(B_\rho,\C))$ 
and, possibly extending $F:B_\rho^3 \times B_\rho \to \C$ to a smooth function defined on the whole $\C^3 \times \C$ 
(which we still denote $F$), $f$ can also be extended to a function defined 
on the whole $\C$ in such a way that $f \in E_T$ and $\norm{f}_{E_{T}} \leq M$, 
with $M$ depending only on $\xi_j, \bra{w_j}_{C^{1/2}}, j=1,2,3$, and $\rho$.

Therefore, by \autoref{prop:existenceexternal} and \autoref{prop:dependenceonf_z} there exists $T^*>0$ sufficiently small such that for every $T<T^*$ it is well defined the map $\Gamma: \mathcal{Z}_T \to \Phi^{-1} (\XX_T)$ that associates to every $y \in \mathcal{Z}_T $ the unique solution in $\Phi^{-1} (\XX_T)$ of 
\begin{align*} 
	\dot{\bar{z}}_j &= \frac{1}{2 \pi i} \sum_{k \neq j}\frac{\xi_k}{z_j - z_k}
	+ F(y_1(t),y_2(t),y_3(t);z_j),
\end{align*}
and for every $y,y' \in \mathcal{Z}_T$ it holds $\|\Gamma(y) - \Gamma(y') \|_\infty \leq C T^{1/2} \|y-y'\|_\infty$. Moreover, by calculations similar to those in the proof of \autoref{lem:contGamma} it is not difficult to see that $\Gamma$ actually takes values in $\mathcal{Z}_T$. 

Thus, since $w \in \mathcal{Z}_T$ for $T$ small enough, $\Gamma$ is a continuous function from a non-empty, convex, compact subset of $C((0,T),\C^3)$, and by Schauder's Theorem it admits a fixed point $z \in \mathcal{Z}_T$, which gives a solution of \eqref{eq:vortexboundeddomain}
with the desired properties.
\end{proof}

\begin{rmk}
The argument above can be easily adapted to prove the existence of a burst of point vortices in the two dimensional torus $\mathbb{T}^2 = \mathbb{R}^2 / \mathbb{Z}^2$. Indeed, the Green function associated with the Laplace operator on the torus with periodic boundary conditions and zero average is given by
\begin{align*}
G_{\mathbb{T}^2}(x,y) = G(x,y) + \gamma_{\mathbb{T}^2}(x,y),
\end{align*}
with $\gamma_{\mathbb{T}^2}(x,y)$ enjoying the same regularity as in the case of bounded domains.
\end{rmk}

\section{Three-Vortices Bursts in a Configuration of Many Vortices}\label{sec:burstN}

In this section we prove that, given \emph{any} configuration of point vortices 
in the plane $\C$ with prescribed vorticities, 
it is always possible to find a solution of the point vortices system
in which three vortices burst out of a single one. 
As already pointed out above, by translation invariance of \eqref{eq:freevortices},
one can suppose without loss of generality that the bursting vortex 
is initially placed at $z_0 = 0$, which we will assume throughout this section.

Let us introduce some notation in order to distinguish the bursting vortices from the others. 
Let $y(0) \in \C^N\setminus \triangle^N$ be a configuration of $N$ vortices with 
$y_k(0) \neq 0$ for every $k=1,\dots,N$ and intensities $\zeta_k \in \R$. 
In the origin $0 \in \C$ we consider a point vortex with intensity 
$\xi \in \R$, $\xi \neq 0$, which is the one we are going to split into three vortices 
of intensities $\xi_1=- \frac{1}{3} \xi$ and $\xi_2=\xi_3 =\frac{2}{3} \xi $ for positive times $t>0$. 
In particular, we will prove the following:

\begin{thm} \label{thm:Nvortices}
Let $y(0) \in \C^N\setminus \triangle^N$, $\zeta_k$, $k=1,\dots,N$ and $\xi$, $\xi_j$, $j=1,2,3$ as above. Then there exists $T>0$ and a solution $(z,y):(0,T) \to \C^3 \times \C^N$ of the full system:
\begin{align} \label{eq:Nvorticesz}
	\dot{\bar{z}}_j	
	&= \frac{1}{2 \pi i} \sum_{\ell \neq j} \frac{\xi_\ell}{z_j - z_\ell}
	+\frac{1}{2 \pi i} \sum_{h = 1}^N \frac{\zeta_h}{z_j - y_h}, \quad j=1,2,3,	\\
	\label{eq:Nvorticesy}
	\dot{\bar{y}}_k	
	&= \frac{1}{2 \pi i} \sum_{\ell=1}^3 \frac{\xi_\ell}{y_k - z_\ell}
	+ \frac{1}{2 \pi i} \sum_{h \neq k} \frac{\zeta_h}{y_k - y_h},
	\quad k=1,\dots,N, 
\end{align}
such that $\lim_{t\to 0} z_j(t) = 0$, $j=1,2,3$, and $\lim_{t \to 0} y_k(t) =y_k(0)$, $k=1,\dots,N$.
\end{thm}

The proof of \autoref{thm:Nvortices} will make use 
of an idea similar to that of \autoref{prop:existenceboundeddomain}:
we will interpret the contribution due to vortices $y_k$ 
in the dynamics of $z_j$ as an external field acting on the bursting vortices only, 
so that the desired existence follows from the results of the previous section. 
However, as in the proof of \autoref{prop:existenceboundeddomain}, 
the external field depends on the position of vortices $z_j$ themselves 
and the argument requires some preliminary work.

Our first step is to ensure that , for small times $T$, 
the vortices $y_k$, $k=1,\dots,N$ do not travel far away from their initial positions $y_k(0)$:
this implies that they generate a sufficiently regular external field on the vortices $z_j$, $j=1,2,3$, 
which will be used in \autoref{ssec:fixedpointN} to prove \autoref{thm:Nvortices}.

\begin{lemma} \label{lem:apriori}
	For every initial configuration $y(0) \in \C^N\setminus \triangle^N$ of the $N$-vortices, with $y_k(0) \neq 0$ for every $k=1,\dots,N$ and for every $\rho>0$ sufficiently small there exists $T^*>0$ such that, if $(z,y):(0,T^*) \to \C^3 \times \C^N$ is a solution to the $(3+N)$-vortices dynamics \eqref{eq:Nvorticesz}, \eqref{eq:Nvorticesy} on $(0,T^*)$ satisfying 
	\begin{gather*}
		\sup_{t<T^*} \max_{j} |z_j(t)| \leq \rho, 
\\
		\lim_{t \to 0} y_k(t) = y(0)
		\quad
		\forall k=1,\dots,N,
	\end{gather*} 
	then
	\begin{align} \label{eq:smalltravel}
		\sup_{t < T^*}
		\max_{k} |y_k(t)-y_k(0)|
		\leq \rho.
	\end{align}
\end{lemma}

\begin{proof}
Set $\rho_1 = \frac{1}{3}
\min_{k \neq \ell} |y_k(0)-y_\ell(0)| $, $\rho_2 = \frac{1}{3}
\min_{k} |y_k(0)-0| $ and take $\rho \leq \min \rho_1,\rho_2$.
Let $(z,y):(0,T) \to \C^3 \times \C^N$ be a solution of \eqref{eq:Nvorticesz}, \eqref{eq:Nvorticesy} on time interval $(0,T)$ satisfying the hypotheses above. 
Now we prove that the restriction $(z,y)$ to $(0,T^*)$ satisfies \eqref{eq:smalltravel} for $T^*$ sufficiently small. Take 
\begin{align*}
\tau = \inf \left\{ s>0: \sup_{t < s}
		\max_{k} |y_k(t)-y_k(0)|
		\leq \rho \right\}.
\end{align*} 
Of course $(z,y)$ satisfies  \eqref{eq:smalltravel} up to time $\tau$, so we only need to prove that $\tau$ is bounded from below by a constant independent of $(z,y)$.
However, by \eqref{eq:Nvorticesz}, \eqref{eq:Nvorticesy} the following bound on $\dot{y}_k$ holds for every $k=1,\dots,N$ and $t < \tau$:
\begin{align*}
|\dot{y_k}(t)| \leq \frac{(N+2)\max_{j,\ell} |\xi_j|,|\zeta_\ell|}{2 \pi \rho},
\end{align*}
so that $\tau \geq T^* = \frac{2 \pi \rho^2}{(N+2) \max_{j,\ell}  |\xi_j|,|\zeta_\ell|}$.
\end{proof}

\subsection{\texorpdfstring{Fixed Point for the $N$-vortices system}{Fixed Point for the N-vortices system}} \label{ssec:fixedpointN}
We can now prove our main result, \autoref{thm:Nvortices}.
The idea is similar to that of \autoref{prop:existenceboundeddomain}, 
and it relies on the results of \autoref{sec:burstinexternal} and Schauder fixed point Theorem.

First of all, we introduce suitable spaces in which to set the fixed point argument.
Take $\rho$ sufficienlty small and $T^*$ given by
\autoref{prop:existenceexternal}, \autoref{prop:dependenceonf} and \autoref{lem:apriori}. For every $T < T^*$ define $\XX_T$ as in \autoref{sec:burstinexternal}, and
\begin{align}\nonumber
	\mathcal{Y}_{\rho,y_0,T} &:= \left\{ y \in \lip((0,T),\C^N):\right.\\
	\label{eq:bc_y}
	&\qquad \lim_{t \to 0 } y_k(t) = y_k(0);\\
	\label{eq:distance_y}
	&\qquad |y_k(t) - y_k(0)| \leq \rho \quad \forall t \in (0,T);\\ 
	\label{eq:holder_y}
	&\qquad \left.
\bra{y}_{\lip} \leq
\frac{(N+2)\max_{j,\ell} |\xi_j|,|\zeta_\ell|}{2 \pi \rho} \quad
\right\}.
\end{align}
The set $\mathcal{Y}_{\rho,y_0,T}$ is convex and compact in $C((0,T),\C^N)$,
the latter endowed with the supremum norm, and it is not empty since 
the constant $y_k(t) \equiv y_k(0)$ belongs to it. 

In order to prove \autoref{thm:Nvortices}, we will regard our equations
as a fixed point of a continuous function on the space $\XX_T \times \mathcal{Y}_{\rho,y_0,T}$. 
We denote by
\begin{align} \label{eq:fNvortices}
f(t,p) = F(y_1(t),\dots,y_N(t);p) =
\frac{1}{2\pi i} \sum_{k=1}^N
\frac{\zeta_k}{p-y_k(t)},
\quad p \in \C,
\end{align}
the field acting on $z_j$ due to the presence of vortices $y_k$.
As already done in the proof of \autoref{prop:existenceboundeddomain}, we can modify $F$ smoothly outside the region of space $B_\rho^N \times B_\rho$ in such a way that, by \autoref{lem:apriori}, for every $y \in \mathcal{Y}_{\rho,y_0,T}$ it holds $f \in E_T$ and $\norm{f}_{E_T} \leq M$, 
with $M$ depending only on $N$, $\xi_j$, $j=1,2,3$, $\zeta_k$, $k=1,\dots,N$ and $\rho$.
It is worth noticing that this modification does not affect the dynamics of vortices $z_j$ for $T<T^*$.

Analogously, we introduce the vector field acting on $y_k$, $k=1,\dots,N$:
\begin{align*}
g(t,p) = G(z_1(t),z_2(t),z_3(t);p) =
\frac{1}{2\pi i} \sum_{j=1}^3
\frac{\xi_j}{p-z_j(t)},
\quad p \in \C.
\end{align*}

Define $\Gamma^\mathcal{X}: \mathcal{Y}_{\rho,y_0,T} \to \XX_T$ 
to be the map that associates to every $y \in \mathcal{Y}_{\rho,y_0,T}$
the unique solution $x \in \XX_T$ of \eqref{eq:vortexternalfield} 
given by \autoref{prop:existenceexternal} and \autoref{cor:uniqueness}, 
with external field $f = F(y)$ given by \eqref{eq:fNvortices} above.

Notice that there exists a constant $C$, depending only on $N$, $\xi_j$, $j=1,2,3$, $\zeta_k$, $k=1,\dots,N$, $\zeta_k$ and $\rho$, such that for every $y,y' \in  \mathcal{Y}_{\rho,y_0,T}$ and $|p| \leq \rho$ it holds 
\begin{align*}
\left| F(y_1(t),\dots,y_N(t);p) - F(y'_1(t),\dots,y'_N(t);p) \right|\leq C |y(t) - y'(t)|.
\end{align*}
Hence, by \autoref{prop:dependenceonf} the map $\Gamma^\mathcal{X}$ is Lipschitz continuous:
\begin{align*}
\|\Gamma^\mathcal{X}(y) - \Gamma^\mathcal{X}(y') \|_\infty \leq C T^{1/2} \|y-y'\|_\infty.
\end{align*}

Let now $\Gamma^\mathcal{Y}: \XX_T \times \mathcal{Y}_{\rho,y_0,T} \to C((0,T),\C^N)$ 
be the map that associated to every $(x,y) \in \XX_T \times \mathcal{Y}_{\rho,y_0,T}$
the curve $\tilde{y}$ of components
\begin{align*}
\tilde{y}_k(t) =
\int_0^t \left( \frac{1}{2\pi i} \sum_{\ell \neq k}
\frac{\zeta_\ell}{y_k-y_\ell}  + G(\Phi^{-1}(x(s));y_k) \right) ds.
\end{align*}
Arguing as in \autoref{lem:contGamma}, one can prove that $\Gamma^\mathcal{Y}$ takes values in $\mathcal{Y}_{\rho,y_0,T}$ and $\Gamma^\mathcal{Y}: \XX_T \times \mathcal{Y}_{\rho,y_0,T} \to \mathcal{Y}_{\rho,y_0,T}$  is continuous with respect to the supremum norm for $T$ small enough.
As a corollary, we get:
\begin{lemma} \label{lem:GammaXYcont}
	For $T>0$ sufficiently small, the map $\Gamma:\XX_T \times \mathcal{Y}_{\rho,y_0,T} \to \XX_T \times \mathcal{Y}_{\rho,y_0,T}$ given by
\begin{align*}
	\Gamma(x,y) = (\Gamma^\mathcal{X}(y),\Gamma^\mathcal{Y}(x,y))
\end{align*}	 
is continuous with respect to the supremum norm.
\end{lemma}

With this at hand, we can finally conclude the proof of our main result.
\begin{proof}[Proof of \autoref{thm:Nvortices}]
Take $\rho$ and $T$ as above.
Since $\XX_T \times \mathcal{Y}_{\rho,y_0,T}$ is a non-empty, compact, convex subset of the Banach space $C((0,T),\C^3 \times \C^N)$ and $\Gamma:\XX_T \times \mathcal{Y}_{\rho,y_0,T} \to \XX_T \times \mathcal{Y}_{\rho,y_0,T}$ is continuous by \autoref{lem:GammaXYcont}, by Schauder fixed point Theorem there exists a fixed point $(x,y) \in \XX_T \times \mathcal{Y}_{\rho,y_0,T}$ for $\Gamma$. Defining $z = \Phi^{-1}(x)$, it is immediate to check that $(z,y)$ is indeed a solution of \eqref{eq:Nvorticesz}, \eqref{eq:Nvorticesy} satisfying $\lim_{t\to 0} z_j(t) = 0$, $j=1,2,3$, and $\lim_{t \to 0} y_k(t) =y_k(0)$, $k=1,\dots,N$.
\end{proof}

\begin{rmk}
Combining the ideas of \autoref{prop:existenceboundeddomain} and \autoref{thm:Nvortices}, one could prove, for any initial configuration of vortices and any arbitrary domain $D$, the existence of a solution to the $N$-vortices system on $D$ with one vortex splitting into three for positive times. The proof basically consists in writing both the boundary term $\nabla^\perp_x \gamma_D$ and field generated by the other vortices $F(y_1(t),\dots,y_N(t);p)$ as an external field on the bursting vortices $z_j$, $j=1,2,3$. We omit the remaining details.
\end{rmk}

\section{Weak Solutions to Euler's Equations and their Intrinsic Stochasticity}\label{sec:weaksol}

We have established in \autoref{thm:Nvortices} existence of vortices bursts for small time intervals.
By inverting time $t\mapsto -t$ and changing signs of intensities $\xi,\xi_j \mapsto -\xi,-\xi_j$,
our result proves also existence of vortices configurations collapsing in finite time
with prescribed intensities and positions at the time of collapse.
With this result at hand, we can thus produce a variety of singular solutions.

In this section, we describe in what sense the point vortices system can be
seen as a weak solution to Euler's equation, both in presence and absence of collapses and bursts,
and then discuss how constructions made possible by \autoref{thm:Nvortices}
make weak solutions not unique.

\subsection{Vortices as Curves in Configuration Space}\label{ssec:configspace}
First and foremost, we need to specify the space in which to set our notion of weak solution.
The natural choice is
\begin{equation*}
\Gamma=\bigcup_{N\geq 0}\Gamma_N, 
\quad \Gamma_N=\set{\gamma=\sum_{j=1}^N \xi_j \delta_{z_j}:\xi_j\in\R, z_j\in\R^2, z_j\neq z_k\text{ if }j\neq k},
\end{equation*}
the \emph{configuration space} of point vortices on $\R^2$, 
to be regarded as a subset of the space of finite signed measures $\M(\R^2)=C_0(\R^2)^*$ endowed with the weak* topology:
this is the space to which the empirical measure $\sum_{j=1}^{N}\xi_j \delta_{z_j}$
of point vortices belong. 
We denote as usual by $|\mu|=\mu^++\mu^-$ the variation of $\mu\in\M(\R^2)$ (in terms of Hahn decomposition).

Since we want to regard point vortices systems as curves in $\Gamma$,
the choice of weak* topology is somewhat forced:
continuous curves $t\mapsto \mu_t$ in $\Gamma$ endowed with the total variation norm $\norm{\mu}=\abs{\mu}(\R^2)$
are such that $t\mapsto \mu_t(\set{x})$ is continuous for every point $x\in\R^2$,
so we can not discuss a system of moving particles in such setting.
To lighten notation, we will denote by $I=(a,b)$ any open interval containing $0$, possibly with $a=-\infty$ or $b=+\infty$.
The following is a trivial, but nonetheless crucial observation for what follows.

\begin{rmk}\label{rmk:extension}
	Let $\xi_1,\dots,\xi_N\in\R$ and $z_1,\dots, z_N\in\R^2$ be given, 
	and consider a solution $I\ni t\mapsto (z_1(t),\dots, z_N(t))\in \R^{2 \times N} \setminus\triangle^N$ 
	of \eqref{eq:freevortices}.
	The empirical measure $I\ni t\mapsto\omega_t=\sum_{j = 1}\xi_j \delta_{z_j(t)}\in\Gamma_N$
	is clearly a weakly* continuous curve.
	
	Consider now the case in which some vortices, say the ones with $j=1,\dots,k$,
	collapse at the right extreme $b$ of $I=(a,b)$, while all others have distinct limits.
	If we consider the configuration at time $b$
	of $N-k+1$ vortices of positions $y_1=\lim_{t \to b} x_1(t)$ and 
	$y_j=\lim_{t \to b} z_{j+k}(t)$ for $j=2,\dots,N-k+1$,
	with intensities $\chi_1=\xi_1+\cdots+\xi_k$ and $\chi_j=\xi_j$ for $j=2,\dots,N-k+1$,
	by local well-posedness there exists a solution 
	$[b,c)\ni t\mapsto (y_1(t),\dots, y_N(t))\in \R^{2 \times N} \setminus\triangle^N$ 
	of the $N-k+1$ vortices system starting from such new configuration.
	The map
	\begin{equation*}
		(a,c)\ni t\mapsto \omega_t=
		\begin{cases}
		\sum_{j = 1}^N\xi_j \delta_{z_j(t)} &t\in(a,b)\\
		\sum_{j = 1}^{N-k+1}\chi_j \delta_{y_j(t)} &t\in [b,c)
		\end{cases}
		\,\in \Gamma
	\end{equation*}
	still defines a weakly* continuous curve.
	The analogue holds for vortices bursting out of the same position at time $a$,
	or for solutions in which distinct groups of vortices collapse or burst
	in different points at the same time.
\end{rmk}

Vortices dynamics that do not include singularities are known to satisfy
Euler's equations \autoref{euler} in a weak form,
but in fact this is still true for solutions ``extended'' in the fashion of \autoref{rmk:extension}. 
The weak formulation we are referring to is based on the following:

\begin{lemma}
	Any smooth solution $\omega\in C^\infty([0,T]\times \R^2)$ of \eqref{euler} satisfies,
	for any $\phi\in C^\infty_c(\R^2)$,
	\begin{align}\label{eq:dssymm}
	\brak{\phi,\omega_t}-\brak{\phi,\omega_0}
	&=\int_0^t \int_{\R^{2\times 2}} H_\phi(x,y)\omega_s(x)\omega_s(y)dxdyds\\ 
	\nonumber
	&=\int_0^t \brak{H_\phi,\omega_s\otimes\omega_s}_{L^2(\R^{2\times 2})}ds,\\
	H_\phi(x,y)&=\frac12 (\nabla\phi(x)-\nabla\phi(y))\cdot K(x-y), \quad x,y\in\R^2,
	\end{align}
	where $H_\phi(x,y)$ is a symmetric function, 
	smooth outside the diagonal set $\triangle^2$ where it has a jump discontinuity
	(regardless of $\phi$),
	and $K$ is the Biot-Savart kernel.
\end{lemma}


\noindent
Here and in what follows, $\brak{\cdot,\cdot}_X$ denotes 
$L^2(X)$-based duality couplings on a measure space $X$. 
Formulation \eqref{eq:dssymm} can be obtained from the usual weak formulation,
\begin{equation*}
\brak{\phi,\omega_t}_{\R^2}-\brak{\phi,\omega_0}_{\R^2}
=\int_0^t \int_{\T^{2\times 2}} K(x-y)\omega_s(y) \omega_s(x) \nabla\phi(x) dxdyds,
\end{equation*}
by symmetrising in variables $x,y$, keeping in mind that $K$
is a skew-symmetric function. In doing so, the singularity of $K(x-y)$ at the diagonal is compensated by $\nabla\phi(x)-\nabla\phi(y)$, so that $H_\phi$ is a bounded function.

This ``symmetrised'' weak vorticity formulation dates back to the works of Delort and Schochet
\cite{Delort91,Schochet95}, and was introduced to give meaning to Euler's equations
in regimes of low space regularity. Indeed, by interpreting brackets $\brak{\cdot,\cdot}$ 
in a suitable way, empirical measures of point vortices systems actually satisfy this formulation.

Given a configuration $\gamma$, we set
\begin{equation*}
	\gamma\diamond\gamma=\sum_{j\neq k}^N \xi_j \xi_k \delta_{(z_j,z_k)}\in \M(\R^{2\times 2}).
\end{equation*}

\begin{definition}\label{def:pvsol}
	A weakly* continuous curve $I\ni t\mapsto \omega_t\in \Gamma$ 
	is a \emph{point vortices solution}
	of Euler's equations if it satisfies for all $t\in I$
	\begin{equation}\label{eq:weaksol}
		\brak{\phi,\omega_t}-\brak{\phi,\omega_0}
		=\int_0^t \brak{H_\phi,\omega_s\diamond\omega_s}_{\R^{2\times 2}}ds.
	\end{equation}
\end{definition}
\noindent
It is immediate to observe that the curves $\omega_t$ in $\Gamma$ described in \autoref{rmk:extension}
satisfy the latter definition.

\begin{rmk}
	Equation \eqref{eq:weaksol} is in fact equivalent to the continuity equation
	\begin{equation*}
		\partial_t\omega_t+\div(v_t\omega_t)=0,
	\end{equation*}
	in distributional sense, where the vector field
	\begin{equation*}
		v_t(z_j)=\sum_{k\neq j}\xi_j K(z_j-z_k), \quad \text{ if }\quad  \omega_t=\sum_{j = 1}^N\xi_j \delta_{z_j},
	\end{equation*}
	(notice that we only need to give the value of $v_t$ only on the support of $\omega_t$)
	is \emph{not} the full fluid dynamical velocity field $u_t=K\ast \omega_t$,
	which would be singular in the support of $\omega_t$, but only retains at point $z_j$
	the influence of \emph{other} vortices.
\end{rmk}

One might extend the definition to weakly* continuous curves $t\mapsto \mu_t\in\M(\R^2)$,
by replacing the double space integral in \eqref{eq:dssymm} with
\begin{equation}\label{eq:diagmisure}
	\int_{\R^{2\times 2}\setminus \triangle^2} H_\phi(x,y)d\mu_t(x)d\mu_t(y).
\end{equation}
The work \cite{Schochet95} considered measure-valued solutions absolutely continuous
with respect to Lebesgue's measure, whose density is of class $H^1(\R^2)$.
For this kind of measures, the diagonal contribution in the double integral against $H_\phi$
is null, so there is no need of an explicit "renormalisation".
Ignoring the diagonal, ``self interaction'' part is the correct way to interpret
the weak formulation \eqref{eq:dssymm} in the case of point vortices \cite{Schochet96}.
Let us also refer again to results of \cite{Marchioro88} and subsequent works obtaining point vortices systems
as limits of more regular solutions of Euler's equations. 

Discussing whether \eqref{eq:diagmisure} provides a good notion of weak solutions
to Euler's equations in the general measure-valued case is out of the scope of this article.
Let us only remark that an analogous situation is the one of Gaussian invariant measures of Euler's equations:
in that case, $\brak{H_\phi,\omega\otimes\omega}$ has to be interpreted as a double It\=o-Wiener integral:
we refer to \cite{Grotto2020b} for a thorough discussion and a link between that context and the present one.

We have observed that \eqref{eq:dssymm} is satisfied by empirical measures of vortices if
the double space integral against $H_\phi$ neglects the contribution of the diagonal
(where $H_\phi$ is discontinuous), that is, self-interaction of vortices:
as a coherence result let us now show that if we fix the number of particles
\autoref{def:pvsol} is satisfied \emph{only} by the empirical measure of a vortex system.

\begin{prop}\label{prop:weaksolunique}
	Let $N\in\N$ be fixed, and consider a point vortices solution $\omega:I\to \Gamma_N$.
	There exists a solution $I\ni t\mapsto (z_1(t),\dots z_N(t))\in \R^{2 \times N}\setminus \triangle^N$ of \eqref{eq:freevortices}
	with vortices intensities $\xi_1,\dots \xi_N$ such that $\omega_t=\sum_{j =1}^N \xi_j \delta_{z_j(t)}$.
\end{prop}

\begin{proof}
	Set $\omega_t=\sum_{j=1}^N \xi_j(t) \delta_{z_j(t)}$.
	Consider $N$ test functions $\phi_j\in C_c^\infty(\R^2)$
	whose disjoint supports are closed balls $B_j(z_j(0),r_j)$ of centres $z_j(0)$ and radii $r_j>0$.
	By weak continuity of $\omega$ there exists a non empty interval $I'\subset I$ such that
	$\brak{\phi_j,\omega_t}>0$ for all $j$ and $t\in I'$.
	In particular, since $\omega_t\in \Gamma_N$, 
	for $t\in I'$, the support of every $\phi_j$ contains one and only one point of
	the support of $\omega_t$:
	since configurations of $\Gamma_N$ are defined up to relabelling,
	we can assume that it is $z_j(t)\in B_j(z_j(0),r_j)$.
	
	Since radii $r_j$ above are arbitrarily small,
	we have that for any $r_j$ small enough there exists a neighbourhood of $t=0$
	for which $\abs{z_j(t)-z_j(0)}\leq r_j$, which means that   
	the map $t\mapsto z_j(t)$ is continuous in $t=0$.
	In fact, the same reasoning can be repeated for any $t\in I'$,
	so $I'\ni t\to z_j(t)\in\R^2$ are continuous functions.
	
	By considering test functions $\phi_j$ taking constant value $1$ on $B(z_j(0),r_j)$
	and $0$ on $B(z_k(0),r_k)$, $k\neq j$, one then shows that $I'\ni t\mapsto \xi_j(t)$
	is constant, since by \eqref{eq:weaksol}
	\begin{align*}
		\xi_j(t)-\xi_j(0)&=\brak{\phi_j,\omega_t}-\brak{\phi_j,\omega_0}\\
		&= \int_0^t\sum_{k\neq h} \xi_h(s) \xi_k(s) \frac{\nabla\phi_j(z_h(s))-\nabla\phi_j(z_k(s))}{2} K(z_h(s)-z_k(s)),
	\end{align*}
	the right-hand side vanishing since $\nabla\phi_j(z_k(s))=0$ for every $k=1,\dots,N$ and $s\in I'$.
	
	Let us now show that points $z_j(t)$ for $t\in I'$ satisfy the point vortices equations.
	By \eqref{eq:weaksol}, considering test functions $\phi_{j,1}$
	such that $\phi_{j,1}(z)=z_1$ (the first component of $z\in\R^2$) on $B(z_j(0),r_j)$
	and $\phi_{j,1}\equiv 0$ on the other $B(z_k(0),r_k)$, $k\neq j$,
	we obtain
	\begin{equation*}
		\xi_j z_{j,1}(t)-\xi_j z_{j,1}(0)
		=\int_0^t \sum_{k \neq j}\xi_j \xi_k K_1(z_j(s)-z_k(s))ds.
	\end{equation*} 
	The analogue can be done for the second component.
	Since points $z_j(t)$ remain far apart for $t\in I'$,
	the above equation actually tells us that they have smooth trajectories,
	and thus they satisfy the differential formulation \eqref{eq:freevortices}.
	
	We have thus proved that the thesis holds for the time interval $I'$:
	let us assume that $I'\subseteq I$ is maximal such that
	for $t\in I'$, $\omega_t=\sum_{j=1}^N \xi_j \delta_{z_j(t)}$
	is the empirical measure of the system of vortices with positions $z_j$ and intensities $\xi_j$
	such that trajectories $z_j(t)$ do not intersect.
	If $I'\neq I$, at a time $\bar t\in \partial I'\cap I$ all limits $\lim_{t \to \bar t} z_j(t)=z_j(\bar t)$
	must exist finite (otherwise violating the hypothesis $\omega_{\bar t}\in\Gamma_N$)
	so either these limits are all distinct --thus we can repeat the argument above
	starting from time $\bar t$ instead of $0$ and extend $I'$, contradicting maximality-- or
	some of them coincide, and in the latter case $\omega_{\bar t}\in \bigcup_{k\leq N-1}\Gamma_k$ contradicts the hypothesis.	
\end{proof}

In contrast to this last result, when the number of vortices is not constrained
we can exploit \autoref{thm:Nvortices} to produce infinitely many solutions.

\begin{thm}
	For any $\gamma=\sum_{j=1}^{N}\xi_j \delta_{z_j}\in\Gamma$ there exist infinite point vortices solutions
	$(-\infty,+\infty)\ni t\mapsto \omega_t \in \Gamma$ in the sense of \autoref{def:pvsol}
	such that $\omega_0=\gamma$.
\end{thm}

\begin{proof}
	One solution, arguably the most ``natural'', is obtained by considering the maximal solution of
	\eqref{eq:freevortices} with intensities $\xi_1,\dots,\xi_N$ and positions $z_1,\dots, z_N$
	(the ones prescribed by $\gamma$) at time $t=0$, letting $\omega_t$ be their empirical measure
	and extending it beyond an eventual collapse time (either for positive or negative times) 
	in the fashion of \autoref{rmk:extension}.
	Notice that this kind of coalescence can only happen a finite number of times,
	since each time the number of vortices is reduced:
	this solution possibly has only one vortex for large times $t\to +\infty$ 
	(into which all vortices have collapsed)
	or for $t\to -\infty$ (so vortices at time $t=0$ were all eventually originated by bursts).
	
	Starting from the solution we just described, let us call it $\omega_t$, we can create infinitely many.
	For instance, we can consider the point vortices solution starting from configuration $\gamma=\omega_0$
	in which one of the vortices $\xi\delta_z$ bursts into three vortices of intensity
	$2\xi/3,2\xi/3,-\xi/3$ as in \autoref{thm:Nvortices}:
	such solution exists by our result at least for a small time interval $[0,T]$,
	but we can continue it until a collapse, and beyond that by making eventual collapsing vortices coalesce.
	The empirical measure is a point vortices solution $\tilde\omega_t$ coinciding with $\omega_t$
	for $t\leq 0$, but differing for $t>0$: indeed in a right interval of $t=0$ the number of
	vortices differ.
	
	In fact, the very same procedure can be applied starting from $\omega_t$ and
	imposing a burst of one of the existing vortices at any other time $t\neq 0$.	
\end{proof}

\subsection{Random Bursts of Vortices and Intrinsic Stochasticity}\label{ssec:intrinsic}
By exploiting non uniqueness of point vortices solutions due to bursts and collapses,
one can produce a $\Gamma$-valued stochastic process whose trajectories
are all point vortices solutions taking the same value $\gamma\in\Gamma$ at time $t=0$.
Here, we limit ourselves to describe one possible construction of this kind,
leaving applications to future works.

Let us thus fix $\gamma\in\Gamma$, and consider a Poisson process $P_\lambda$ of parameter $\lambda>0$,
whose jump times we denote by $t_1,t_2,\dots\in \R^+$, defined on a probability space $(\Omega,\F,\PP)$.
We then define a stochastic process $t\mapsto \omega_t\in\Gamma$ as follows:
\begin{itemize}
	\item on $[0,t_1]$ we consider the solution of point vortices dynamics \eqref{eq:freevortices}
	starting from configuration $\gamma=\sum_{j = 1}^N\xi_j \delta_{z_j}$,
	continuing it beyond eventual collapses by merging vortices, as described above;
	\item let $\omega_{t_1}=\sum_{j = 1}^{N_{t_1}}\xi_j \delta_{z_j(t)}$,
	and (possibly changing the probability space), 
	consider a uniform random $n_1$ variable on $\set{1,\dots, N_{t_1}}$ independent of $P_\lambda$:
	by \autoref{thm:Nvortices} we can consider a solution of \eqref{eq:freevortices}
	with $N_{t_1}+2$ vortices on some small time interval $(t_1,t_1+\delta)$,
	in which for $t\downarrow t_1$ three vortices burst out of $\xi_{n_1}\delta_{z_{n_1}}$;
	the empirical measure of such burst can be the continued until $t_2$ eventually
	merging collapsing vortices;
	\item the above point is repeated for all successive jump times,
	enriching the probability space with new uniform variables on indices of existing vortices
	all independent from previously defined variables.	
\end{itemize}
It is then easy to show that
$[0,\infty)\ni t\mapsto \omega_t\in\Gamma$ thus defined is a Markov process,
and that every sample of it is a point vortices solution of Euler's equation in
the sense of \autoref{def:pvsol}, with $\omega_0=\gamma$.
We might call processes of this kind \emph{intrinsically stochastic} weak solutions
of Euler's equation, or \emph{Markov selections} of weak solutions.

\subsection{Energy Dissipation in Bursts of Vortices}
We conclude our discussion by observing that
burst of vortices we constructed in \autoref{thm:Nvortices} actually \emph{dissipate} energy.
Let us first recall what energy is in the point vortex context.
The usual definition, that is the $L^2$ norm of the velocity field induced by vorticity
$\omega=\sum_{j=1}^N \xi_j \delta_{z_j}$, is not viable:
since $K(x)=-\nabla^\perp(-\Delta)^{-1}$,
a formal integration by parts gives
\begin{align*}
	\int_{\R^2} |u(x)|^2 dx
	&=\int_{\R^2} (K\ast \omega)(x)\cdot (K\ast \omega)(x) dx
	=-\frac{1}{2\pi}\sum_{j,k=1}^N \xi_j\xi_k \log |z_j-z_k|,
\end{align*}
which is a divergent expression. Instead, one must consider
\begin{equation*}
	H(\omega)=-\frac{1}{2\pi}\sum_{j\neq k} \xi_j\xi_k \log |z_j-z_k|,
\end{equation*}
which in fact is the Hamiltonian function of Hamiltonian system \eqref{eq:freevortices}
in conjugate coordinates $(z_{j,1},\xi_j z_{j,2})_{j=1,\dots,N}$.
In other words, one neglects the diverging self-interaction terms,
and only considers interaction between different vortices.
In the notation of \autoref{ssec:configspace}, the energy of a configuration $\gamma\in\Gamma$
is thus suggestively expressed by
\begin{equation*}
	H(\gamma)=\brak{G,\gamma\diamond\gamma}, \quad G(x,y)=-\frac1{2\pi}|x-y|.
\end{equation*}

\begin{prop}
	In the notation of \autoref{thm:Nvortices}, consider for $t\in [0,T]$
	\begin{equation*}
		\omega_t=\sum_{\ell=1}^3\xi_\ell \delta_{z_\ell(t)}+\sum_{k=1}^{N}\zeta_k\delta_{y_k(t)},
	\end{equation*}
	the empirical measure of the vortex burst. Then, $H(\omega_t)$ is constant for $t>0$,
	but $H(\omega_0)>H(\omega_t)$, $t>0$.
\end{prop}

\begin{proof}
	Since $z_1(0)=z_2(0)=z_3(0)=0$, at time $t$ they sum up to a unique vortex with position $0$
	and intensity $\xi_1+\xi_2+\xi_3=\xi$. We thus have
	\begin{equation*}
		H= H(\omega_0) = -\frac1{2\pi} \sum_{k\neq \ell}^N \zeta_k \zeta_\ell \log \left| y_k(0) - y_\ell(0) \right| 
		-\frac1{2\pi} \sum_{k = 1}^N \zeta_k \xi \log \left| y_k(0) \right|
		=H_1 + H_2. 
	\end{equation*}
	For $t\in (0,T]$, we are dealing with a solution of \eqref{eq:freevortices}
	without singularities, so the Hamiltonian is conserved, and we have, for $t>0$,
	\begin{align*}
	H' = H(\omega_t) &= 
	-\frac1{2\pi} \sum_{k\neq \ell}^N \zeta_k \zeta_\ell \log \abs{y_k(t) - y_\ell(t)}
	-\frac1{2\pi} \sum_{k=1}^N\sum_{\ell=1}^3 \zeta_k \xi_j \log \abs{y_k(t) - z_j(t)}\\
	&\quad -\frac1{2\pi} \sum_{j\neq \ell}^3 \xi_j \xi_h \log \left|  z_j(t)- z_h(t) \right|
	= H'_1(t) + H'_2(t) + H'_3(t).
	\end{align*}
	Let us stress the fact that although the summands $H'_1(t)$, $H'_2(t)$, $H'_3(t)$ 
	may depend on $t \in (0,T)$, the total interaction energy $H'$ is time-independent.
	
	It is now easy to check that $\lim_{t \to 0} H'_1(t) = H_1$ and $\lim_{t \to 0} H'_2(t) = H_2$, so
	\begin{align*}
	H' - H = \lim_{t \to 0} H'_3(t).
	\end{align*}
	Using the relations $\xi_1 = -\frac{1}{3} \xi$, $\xi_2=\xi_3=\frac{2}{3} \xi $,  and $z_j(t) = z_1(t) \left( x_j(t) + \frac{a_j}{a_1 }\right)$, $j=2,3,$ we rewrite
	\begin{align*}
	H'_3(t)
	&= 
	\frac{2\xi^2}{9}  
	\log \left| a_1 x_2(t) + a_2 - a_1  \right|
	+ \frac{2\xi^2}{9} 
	\log \left|  a_1 x_3(t) + a_3 - a_1  \right|
	\\
	&\quad
	- \frac{4\xi^2}{9}
	\log \left|  a_1 x_2(t) - a_1 x_3(t) + a_2-a_3  \right|,
	\end{align*}
	and therefore, by using the explicit choice of numbers $a_i$ made in the proof
	of \autoref{thm:Nvortices} (see \autoref{appendix}), we have
	\begin{align*}
	H' - H &=
	\frac{2 \xi^2}{9}
	\left( 
	\log \left| a_2 - a_1  \right|
	+ \log \left|  a_3 - a_1  \right|
	-2 \log \left|  a_2-a_3  \right| \right)
	\\
	&=
	\frac{\xi^2}{9}
	\left( 
	\log 3 + \log 21 - 2 \log 12 \right) < 0.\qedhere
	\end{align*}	
\end{proof}

As a consequence, if we continue the point vortices evolution in a small left neighbourhood of $t=0$,
by considering the three vortices in $0$ at $t=0$ as a single one before that time,
we have produced a point vortices solution whose energy is discontinuous
at $t=0$, and it decreases following a burst.
By reversing time, this corresponds to an energy increase in correspondence of a point vortices collapse.
In any case, arguments on energy balance --that is, asking where does dissipated energy goes--
are difficult to set up, since it does not seem meaningful to compare the
infinite self-interaction energy of the bursting vortex to the finite loss of energy after burst.

The latter result does not prove that \emph{any} burst or collapse of vortices
provokes a discontinuity in the otherwise constant energy,
the computation relying on the explicit construction of previous sections.
Nevertheless, we might ask whether point vortices solutions $t\mapsto \omega_t$
in the sense of \autoref{def:pvsol} are unique if we further impose that
$t\mapsto H(\omega_t)$ is constant.
Since conservation of energy $H$ does not prevent vortex collapse or burst
--indeed, the dissipation we observe is a consequence of our choice
in continuing the solution for $t\leq 0$--
it is not clear how to apply energy conservation,
so we leave this question open for future works.

\appendix

\section{Proofs of Technical Lemmas}\label{appendix}

\subsection*{Proof of \autoref{lem:existselfsimilar}}

First of all, we discuss the existence of  suitable parameters $a>0$, $b\in\R$, $a_1,a_2,a_3\in\C$ as in the statement of the lemma.
The three equations of \eqref{eq:asrelation} (for $j=1,2,3$)
are not linearly independent, and setting for brevity $\chi=\frac{6\pi i}{\xi}(a-ib)$,
they provide two independent relations between $a_1,a_2,a_3$:
\begin{align*}
	\chi\bar a_2 &=\pa{\frac2{a_2-a_3}-\frac1{a_2-a_1}},\\
	\chi\bar a_3 &=\pa{\frac2{a_3-a_2}-\frac1{a_3-a_1}}.
\end{align*}
We can reduce a variable by imposing that the centre of vorticity is $0=\xi_1z_1+\xi_2z_2+\xi_3z_3$,
since this implies, for our choice of intensities, that $a_1=2(a_2+a_3)$.
Simple passages then lead to an equation involving only $a_2,a_3$:
\begin{equation*}
	|a_2|^2+|a_3|^2+4\re (\bar a_2 a_3)=0.
\end{equation*}
Solutions $a_2,a_3$ to this last equation are of course not unique.
From now on, since our aim is only \emph{existence} of a set of parameters $a,b,a_1,a_2,a_3$
as in the statement, we can impose further conditions to reduce computations to particular,
amenable cases.

Let us thus impose that $a_3=\lambda\in\R$ is real (in fact this is without loss of generality,
by rotation invariance of the system), so that if $a_2=A+iB$ we want to solve
\begin{equation*}
	\lambda^2+A^2+B^2+4\lambda A=0,
\end{equation*}
and from relations above given such $\lambda,A,B$ one can retrieve $a_1$ and
\begin{equation*}
	\chi=\frac{6\pi i}{\xi}(a-ib)=-\frac{3}{\lambda}\cdot \frac{A+i B+\lambda}{(A+i B-\lambda)(2A+2iB+\lambda)}.
\end{equation*}
This last relation is important: since $a>0$, if $\lambda,A,B$ can be chosen so that the imaginary part of $\chi$
can take any value in $\R\setminus\set{0}$, so will $\xi$, as we are requiring.
Elementary computations reveal that the following two sets of parameters
\begin{gather*}
	\xi>0,\,a=\frac{\sqrt{3}}{84\pi}\xi,\,b=\frac{5}{84\pi}\xi,\,
	a_1=-2+i2\sqrt 3,\,a_2=-2+i\sqrt 3,\,a_3=1;\\
	\xi<0,\,a=-\frac{\sqrt{3}}{84\pi}\xi,\,b=\frac{5}{84\pi}\xi,\,
	a_1=2+i2\sqrt 3,\,a_2=2+i\sqrt 3,\,a_3=-1;\\
\end{gather*}
satisfy the statement of the Lemma, allowing a choice of $a,b,a_1,a_2,a_3$ for any $\xi\in \R\setminus\set{0}$. 

Concerning H\"older regularity of such a solution, for $t>s>0$,
\begin{align*}
|Z(t)-Z(s)| \leq &
|Z(t) - \sqrt{2as}\,e^{i \frac{b}{2a} \log t}| + 
|\sqrt{2as}\,e^{i \frac{b}{2a} \log t}-Z(s)|
\\
= &
\sqrt{2at} - \sqrt{2as} 
+
\sqrt{2as}\,
|e^{i \frac{b}{2a} \log t}-e^{i \frac{b}{2a} \log s}|
\\
\leq &
\sqrt{2a(t-s)}
+
\sqrt{2as}\,
\frac{|b|}{2a}\left( \log t - \log s \right)
\\
=&
\sqrt{2a(t-s)}
+
\sqrt{2as}\,
\frac{|b|}{2a}
\int_s^t \frac{dr}{r}
\leq 
\sqrt{2a(t-s)}
+
\sqrt{2a}\,
\frac{|b|}{2a}
\int_s^t \frac{dr}{r^{1/2}}
\\
= &
\sqrt{2a(t-s)}
+
\sqrt{2a}\frac{|b|}{2a} \frac{\sqrt{t}-\sqrt{s}}{2}
\leq 
\sqrt{2 a} \left(1+\frac{|b|}{4a}\right) \sqrt{t-s},
\end{align*} 
concluding the proof.
 
\subsection*{Proof of \autoref{lem:changeofcoor}: Expansions and Leading Orders}
Let the notation established in \autoref{sec:burstinexternal} prevail.
In order to express the rational functions of $z_1,z_2,z_3$ in terms of the new coordinates,
we will make use of elementary Taylor expansions.

Let us consider 
\begin{equation*}
Q_j(x)=\frac1{1-x-{a_j}/{a_1}},
\quad x \in \C, \quad j=2,3,
\end{equation*}
for small values of $|x|$. Since $a_j \neq a_1$, $Q_j$ is holomorphic for $|x|<\rho_j= |1-a_j/a_1|/2$,
and for the latter values $Q_j$ has a convergent power series expansion $Q_j(x)=\sum_{n=0}^{\infty} c_{j,n} x^n$.
Remainders of the series,
\begin{equation}\label{eq:taylorQk}
R_{j,m+1}(x)
=Q_j(x) - \sum_{n=0}^{m} c_{j,n} x^n 
=\sum_{n=m+1}^{\infty} c_{j,n} x^n,
\end{equation}
are holomorphic functions satisfying, for $|x|<\rho_j / 2$,
\begin{equation} \label{eq:estRk}
|R_{j,m+1}(x)| \leq 2 \rho_j^{-m-2} |x|^{m+1},\quad 
|R'_{j,m+1}(x)| \leq C_m \rho_j^{-m-2}|x|^{m},
\end{equation}
$C_m>0$ depending only on $m$. Similar estimates hold for the functions
\begin{align*}
Q_{j,k}(x) = \frac1{x+\frac{a_j}{a_1}- \frac{a_k}{a_1}}, \quad |x|<|a_j-a_k|/{4|a_1|},
\end{align*}
we denote by $R_{j,k,m}$ remainders of $Q_{j,k}$.

The dynamics of $r$ and $\theta$ are the simpler ones. For the former we have
\begin{align*}
\frac{d}{dt}(r^2)
=&\frac{1}{|a_1^2|} \frac{d}{dt}(z_1 \bar{z}_1)
=\frac{2}{|a_1^2|} \text{Re} \pa{\frac{z_1}{2\pi i} \sum_{k \neq 1}
	\frac{\xi_k}{z_1-z_k} + z_1 f(t,z_1)}
\\
=&
\frac{2}{|a_1^2|} 
\text{Re} \left(
\frac{1}{2\pi i} \sum_{k \neq 1}
\frac{\xi_k}{1-x_k-\frac{a_k}{a_1}} \right)
+
\frac{2}{|a_1^2|} 
\text{Re} \left(
z_1 f(t,z_1) \right).
\end{align*}
Recalling \eqref{eq:taylorQk} and making use of \eqref{eq:asrelation}, we rewrite the expression above as 
\begin{align*}
\frac{d}{dt}(r^2)
=&
\frac{2}{|a_1^2|} 
\text{Re} \pa{
	\frac{1}{2\pi i} \sum_{k \neq 1}
	\frac{a_1\xi_k}{a_1-a_k}
	+
	\frac{1}{2\pi i} \sum_{k \neq 1}
	R_{k,1}(x_k) }
+
\frac{2}{|a_1^2|} 
\text{Re} \left(
z_1 f(t,z_1) \right)
\\
=&
2a +
\frac{2}{|a_1^2|} 
\text{Re} \pa{
	\frac{1}{2\pi i} \sum_{k \neq 1}
	R_{k,1}(x_k) }
+
\frac{2}{|a_1^2|} 
\text{Re} \left(
z_1 f(t,z_1) \right)
\\
=&
2a +
\omega_{r}(x_2,x_3) +
\frac{2}{|a_1^2|} 
\text{Re} \left(
z_1 f(t,z_1) \right),
\end{align*}
where by \eqref{eq:estRk}, there exist $C,\rho'>$ depending only on $\xi,a_1,a_2,a_3$ such that
\begin{align*}
\forall |x_2|, |x_3| < \rho',\quad 
|\omega_{r}(x_2,x_3)| \leq C \left(|x_2| + |x_3|\right),\quad
|\nabla \omega_{r}(x_2,x_3)| \leq C.
\end{align*}
By an analogous computation,
\begin{align*}
\frac{d}{dt}\theta&=-i \frac{d}{dt} \log\pa{\frac{z_1}{a_1 r}}
=-i\frac{r}{z_1} \frac{d}{dt}\pa{\frac{z_1}{r}}\\
&=\im \pa{\frac{\dot z_1}{z_1}-\frac{\dot r}{r}}=\im \pa{\frac{\dot z_1}{z_1}}\\
&=-\im \pa{ \frac{1}{2\pi i \bar z_1} \sum_{k \neq 1}
	\frac{\xi_k}{z_1-z_k} }+
\im\pa{ \frac{\overline{f(t,z_1)}}{z_1} }\\
&=\frac{b}{r^2}
-\im \pa{ \frac{|a_1|^2}{2\pi i r^2} \sum_{k \neq 1}
	R_{k,1}(x_k)}
+\im\pa{ \frac{\overline{f(t,z_1)}}{z_1} }
\\
&=\frac{b}{r^2}
+ \frac{\omega_{\theta}(x_2,x_3)}{r^2}
+\im\pa{ \frac{\overline{f(t,z_1)}}{z_1} },	
\end{align*}
where we have used the fact that $\re\pa{\frac{\dot z_1}{z_1}}=\frac{\dot r}{r}$
(which is easily verified by the definitions). Here, again by \eqref{eq:estRk}, $\omega_{\theta}$ satisifies
\begin{align*}
\forall |x_2|, |x_3| < \rho', \quad 
|\omega_{\theta}(x_2,x_3)| \leq C \left(|x_2| + |x_3|\right),\quad
|\nabla \omega_{\theta}(x_2,x_3)| \leq C,
\end{align*}
possibly reducing constants $C,\rho'>0$ depending again only on $\xi,a_1,a_2,a_3$.

The evolution of $x_j$, $j=2,3,$ requires more care. It holds
\begin{align*}
\frac{d}{dt} x_j &= \frac{1}{z_1^2}\pa{z_1 \frac{d}{dt} z_j - z_j \frac{d}{dt} z_1}\\
&=\frac{1}{z_1} \overline{\pa{\frac{1}{2\pi i} \left(
		\frac{\xi_k}{z_j-z_k}+
		\frac{\xi_1}{z_j-z_1} \right)+ f(t,z_j)}}\\
&\qquad-\frac{z_j}{z_1^2}\overline{\pa{\frac{1}{2\pi i} \sum_{\ell \neq 1}
		\frac{\xi_\ell}{z_1-z_\ell} + f(t,z_1)}}.
\end{align*}
Let us first study the part of vector field due to vortices interactions:
we have
\begin{align*}
\frac{1}{2\pi i}\pa{\frac{\xi_k}{z_j-z_k}+ \frac{\xi_1}{z_j-z_1}}
=&
\frac{1}{2\pi i z_1} \left(
\frac{\xi_k}{x_j-x_k + \frac{a_j}{a_1}- \frac{a_k}{a_1}}+
\frac{\xi_1}{x_j+ \frac{a_j}{a_1} -1} \right)
\\
=&
\frac{1}{2\pi i z_1} \left(
Q_{j,k}(x_j-x_k) - \frac{\xi_1}{\xi_j} Q_j(x_j)\right).
\end{align*}
Therefore,
\begin{align*}
&\frac{1}{2\pi i}\pa{\frac{\xi_k}{z_j-z_k}+ \frac{\xi_1}{z_j-z_1}}\\
& \quad=
\frac{1}{2\pi i z_1} \left(
\frac{a_1 \xi_k}{a_j- a_k}
-\frac{a_1^2 \xi_k}{(a_j- a_k)^2}(x_j-x_k)
+R_{j,k,2}(x_j-x_k)\right.\\
&\qquad \left. +\frac{a_1 \xi_1}{a_j- a_1}-
\frac{a_1^2 \xi_1}{(a_j- a_1)^2}x_j
-\frac{\xi_1}{\xi_j}R_{j,2}(x_k)\right)
\\
&=
\frac{1}{2\pi i z_1} \left(
2\pi i \overline{a_j}a_1(a-i b)-
\frac{a_1^2 \xi_k}{(a_j- a_k)^2}(x_j-x_k)-
\frac{a_1^2 \xi_1}{(a_j- a_1)^2}x_j \right.\\
&\qquad \left.
+R_{j,k,2}(x_j-x_k) -\frac{\xi_1}{\xi_j}R_{j,2}(x_k)\right),
\end{align*}
hence
\begin{align*}
\frac{1}{z_1}\overline{\frac{1}{2\pi i} \left(
	\frac{\xi_k}{z_j-z_k}+
	\frac{\xi_1}{z_j-z_1} \right)}
=&
\frac{a_j}{a_1}\frac{|a_1|^2}{|z_1|^2}  (a+ib)
\\
&+
\frac{1}{2\pi i |z_1|^2} \overline{\left(
	\frac{a_1^2 \xi_k}{(a_j- a_k)^2}(x_j-x_k)+
	\frac{a_1^2 \xi_1}{(a_j- a_1)^2}x_j \right)}
\\
&+
\frac{1}{2\pi i |z_1|^2} \overline{\left(
	R_{j,k,2}(x_j-x_k) -
	\frac{\xi_1}{\xi_j}R_{j,2}(x_k) \right)}.
\end{align*}
We move now to the contribution given by $\dot{z}_1$:
\begin{align*}
\frac{1}{2\pi i} \sum_{\ell \neq 1}
\frac{\xi_\ell}{z_1-z_\ell}
=&
\frac{1}{2\pi i z_1}  \left(
\frac{\xi_k}{1-x_k- \frac{a_k}{a_1}}+
\frac{\xi_j}{1-x_j- \frac{a_j}{a_1}} \right)\\
=&
\frac{1}{2\pi i z_1}
\left(
\frac{a_1 \xi_j}{a_1- a_j} + 
\frac{a_1^2 \xi_j}{(a_1- a_j)^2}x_j 
+
\frac{a_1 \xi_k}{a_1- a_k} + 
\frac{a_1^2 \xi_k}{(a_1- a_k)^2}x_k\right.\\
&\quad \left.
+
R_{j,2}(x_j) +
R_{k,2}(x_k)
\right)
\\
=&
\frac{1}{2\pi i z_1}
\left(
2 \pi i |a_1|^2 (a-ib) + 
\frac{a_1^2 \xi_j}{(a_1- a_j)^2}x_j 
+ 
\frac{a_1^2 \xi_k}{(a_1- a_k)^2}x_k\right.\\
&\quad \left.
+
R_{j,2}(x_j) +
R_{k,2}(x_k)
\right),
\end{align*}
therefore
\begin{multline*}
-\frac{z_j}{z_1^2} \overline{
	\left(\frac{1}{2\pi i} \sum_{\ell \neq 1}
	\frac{\xi_\ell}{z_1-z_\ell} \right)}
=
-\frac{z_j}{z_1}\frac{|a_1|^2}{|z_1|^2}  (a+ib)
\\
+\frac{z_j}{z_1} \frac{1}{2\pi i |z_1|^2} \overline{\left(
	\frac{a_1^2 \xi_j}{(a_1- a_j)^2}x_j 
	+ 
	\frac{a_1^2 \xi_k}{(a_1- a_k)^2}x_k
	+
	R_{j,2}(x_j) +
	R_{k,2}(x_k)
	\right)}.
\end{multline*}

All in all, we get
\begin{align*}
\frac{d}{dt} x_j
=&
- x_j \frac{a+ib}{r^2}
\\
&+
\frac{1}{2\pi i |z_1|^2} \overline{\left(
	\frac{a_1^2 \xi_k}{(a_j- a_k)^2}(x_j-x_k)+
	\frac{a_1^2 \xi_1}{(a_j- a_1)^2}x_j \right)}
\\
&+
\left( x_j + \frac{a_j}{a_1} \right) \frac{1}{2\pi i |z_1|^2} \overline{\left(
	\frac{a_1^2 \xi_j}{(a_1- a_j)^2}x_j 
	+ 
	\frac{a_1^2 \xi_k}{(a_1- a_k)^2}x_k
	\right)}
\\
&+
\frac{1}{2\pi i |z_1|^2} \overline{\left(
	R_{j,k,2}(x_j-x_k) -
	\frac{\xi_1}{\xi_j}R_{j,2}(x_k)\right)}
\\
&+
\left( x_j + \frac{a_j}{a_1} \right) \frac{1}{2\pi i |z_1|^2} \overline{\left(
	R_{j,2}(x_j) +
	R_{k,2}(x_k)
	\right)}
\\
&+\frac{1}{z_1} \overline{f(t,z_j)}
-\frac{z_j}{z_1^2}\overline{f(t,z_1)}.
\end{align*}
\noindent
We arrive at the expression in \autoref{lem:changeofcoor}, that is 
\begin{equation*}
\frac{d}{dt}x_j =
\frac{L_{j}(x_2,x_3,\overline{x_2},\overline{x_3})}{r^2} +
\frac{\omega_j(x_2,x_3,x_2-x_3)}{r^2}
+\frac{1}{z_1} \overline{f(t,z_j)}
-\frac{z_j}{z_1^2}\overline{f(t,z_1)}.
\end{equation*}
by collecting linear terms into $L_2,L_3$
and the remainders $R$ into holomorphic functions $\omega_j$ satisifying for all $|x_2|, |x_3|, |x_2-x_3| < \rho'$
\begin{gather*}
|\omega_{j}(x_2,x_3,x_2-x_3)| \leq C \left(|x_2|^2 + |x_3|^2\right),\\
|\nabla \omega_{j}(x_2,x_3,x_2-x_3)|\leq C
\left(|x_2| + |x_3|\right),
\end{gather*}
where we redefined for the last time constants $C,\rho'>0$ depending only on $\xi,a_1,a_2,a_3$.
Notice that the proof has used, up to this point,
only the fact that parameters $a,b,a_1,a_2,a_3$ satisfy \eqref{eq:asrelation}
(rather than the particular choice in the end of proof of \autoref{lem:existselfsimilar}).

\subsection*{Proof of \autoref{lem:changeofcoor}: Eigenvalues}
We are left to prove the statement on eigenvalues of the matrix
\begin{equation*}
L   = \pa{L_2,L_3,\bar L_2,\bar L_3}
	=\pa{\begin{array}{cccc}
	-a-i b 	& 	0 	& L_{13} & L_{14}  	\\
	0 	& 	-a-i b 	& L_{23} & L_{24}  	\\
	\overline{L_{13}} & \overline{L_{14}} & -a + i b & 0						\\
	\overline{L_{23}} & \overline{L_{24}} & 0 	& -a + i b 
	\end{array}},
\end{equation*}
where
\begin{align*}
L_{13} =&
\frac{1}{2\pi i |a_1|^2}
\overline{
	\left(
	\frac{a_1^2 \xi_3}{(a_2-a_3)^2}+
	\frac{a_1^2 \xi_1}{(a_2-a_1)^2}+
	\frac{\overline{a_2}}{\overline{a_1}}\frac{a_1^2 \xi_2}{(a_1-a_2)^2}
	\right)}, \\
L_{14} =&
\frac{1}{2\pi i |a_1|^2}
\overline{
	\left(
	-\frac{a_1^2 \xi_3}{(a_2-a_3)^2}+
	\frac{\overline{a_2}}{\overline{a_1}}\frac{a_1^2 \xi_3}{(a_1-a_3)^2}
	\right)}, \\
L_{23} =&
\frac{1}{2\pi i |a_1|^2}
\overline{
	\left(
	-\frac{a_1^2 \xi_2}{(a_3-a_2)^2}+
	\frac{\overline{a_3}}{\overline{a_1}}\frac{a_1^2 \xi_2}{(a_1-a_2)^2}
	\right)}, \\
L_{24} =&
\frac{1}{2\pi i |a_1|^2}
\overline{
	\left(
	\frac{a_1^2 \xi_2}{(a_3-a_2)^2}+
	\frac{a_1^2 \xi_1}{(a_3-a_1)^2}+
	\frac{\overline{a_3}}{\overline{a_1}}\frac{a_1^2 \xi_3}{(a_1-a_3)^2}
	\right)}.
\end{align*}
The eigenvalues of $L$ coincide with the roots of its characteristic polynomial,
which is given by
\begin{align*}
p(\lambda) &= y^2 - y c_1 + c_2,\\
y  &=(-a-i b -\lambda)(-a+i b -\lambda)
= (a+\lambda)^2 + b^2,\\
c_1&=L_{23}\overline{L_{14}} + 
L_{24}\overline{L_{24}} +
L_{13}\overline{L_{13}} +
L_{14}\overline{L_{23}}, \\
c_2&=L_{13}\overline{L_{13}}L_{24}\overline{L_{24}}+
L_{23}\overline{L_{23}}L_{14}\overline{L_{14}} -
L_{14}\overline{L_{13}}L_{23}\overline{L_{24}} -
L_{13}\overline{L_{14}}L_{24}\overline{L_{23}}.
\end{align*}
Notice that $c_1,c_2 \in \R$. 
We now claim that the eigenvalues have the form $-a + i \mu$, $\mu \in \R$. 
This is true if and only if $y=b^2-\mu^2$ solves
\begin{align*}
y^2 - y c_1 + c_2 = 0,
\end{align*}
which in terms of $\mu$ becomes
\begin{align*}
\mu^4 + \mu^2 (c_1 - 2b^2) + b^4 - b^2 c_1 + c_2  = 0.
\end{align*}
From this, we deduce that if it holds
\begin{equation}\label{eq:conditionmu}
2b^2 - c_1 \pm \sqrt{(2b^2 - c_1)^2 - 4(b^4 - b^2 c_1 + c_2)} > 0,
\end{equation}
there are 4 distinct solutions $\mu$, producing 4 distinct roots of $p$,
and thus all and only the eigenvalues.
\emph{Only at this point of the proof} we restrict ourselves to the particular
choice of parameters made at the end of the proof of \autoref{lem:existselfsimilar}.
Indeed, with such choice, \eqref{eq:conditionmu} is satisfied independently from the sign of $\xi$:
\begin{align*}
2b^2 - c_1 \sim \xi^2 \times 3.4463 \times 10^{-4}, \quad
b^4 - b^2 c_1 + c_2  \sim \xi^4 \times 2.7035 \times 10^{-9},
\end{align*}
and this concludes the proof.

\end{document}